\documentclass[11pt,a4paper]{amsart}
\usepackage{geometry}
\usepackage[english]{babel}
\usepackage{amssymb,enumerate,latexsym,hyperref,color,graphicx,cite}

\newenvironment{proof*}[1]{\medskip\noindent\textbf{#1\ }}{\hspace*{\fill}$\Box$\medskip}
\newtheorem{theorem}{Theorem}[section]
\newtheorem{lemma}[theorem]{Lemma}
\newtheorem{proposition}[theorem]{Proposition}
\newtheorem{corollary}[theorem]{Corollary}
\newtheorem{definition}[theorem]{Definition}
\theoremstyle{remark}
\newtheorem{remark}[theorem]{Remark}
\newtheorem{example}[theorem]{Example}

\newtheorem{question}{\bfseries Question}

\newtheorem{mainthm}{\bf Theorem}

\newtheorem{coralph}[mainthm]{\bf Corollary}

\newcommand{\W}{{\mathbf W}}

\newcommand{\NN}{{\mathbb N}}

\newcommand{\PP}{{\mathbb P}}

\newcommand{\RR}{{\mathbb R}}

\newcommand{\TT}{{\mathbb T}}

\newcommand{\ZZ}{{\mathbb Z}}

\newcommand{\be}[1]{\begin{equation} \label{#1} }
\newcommand{\ee}{\end{equation}}
\newcommand{\beq}{\begin{equation}}

\def \W{{\mathcal W}}

\def \W{\mathcal{W}}

\def \hx0{\hat{x_0}}

\author{Martin Leguil}
\thanks{}
\address{\'Ecole polytechnique, CMLS\\
	Route de Saclay\\ 91128 Palaiseau Cedex\\ France}
\email{martin.leguil@polytechnique.edu}

\author{Disheng Xu}
\thanks{}
\address{School of Science, Great Bay University and Great bay institute for advanced study\\
Songshan Lake International Innovation Entrepreneurship Community A5\\ Dongguan
523000\\ China}
\email{xudisheng@gbu.edu.cn}

\author{Jiesong Zhang}
\thanks{}
\address{Department of mathematics\\ Kungliga Tekniska hogskolan, Lindstedtsvägen 25\\ SE-100 44 Stockholm\\ Sweden}
\email{jiesongz@kth.se}

\begin{document}

\title[The Lipschitz threshold]{Extremal distributions of partially hyperbolic systems: the Lipschitz threshold}

\begin{abstract}
We prove a sharp phase transition in the regularity of the extremal distribution $E^s \oplus E^u$ for $C^\infty$ volume-preserving partially hyperbolic diffeomorphisms on closed $3$-manifolds: if $E^s \oplus E^u$ is Lipschitz, then it is automatically $C^\infty$. This extends the rigidity phenomenon established by Foulon--Hasselblatt~\cite{fh03} for conservative Anosov flows in dimension $3$ to the partially hyperbolic setting.

This gain in regularity has several applications to rigidity problems. In particular, we study the relationship between the $\ell$-integrability condition introduced by Eskin--Potrie--Zhang~\cite{epz} and joint integrability in the conservative setting, yielding rigidity results for $u$-Gibbs measures. We also obtain several $C^\infty$ classification results for partially hyperbolic diffeomorphisms on $3$-manifolds under various assumptions.


\end{abstract}

\maketitle

\tableofcontents

\section{Introduction}
\subsection{A brief introduction to partially hyperbolic systems}
Partially hyperbolic systems arise naturally in smooth dynamics as a robust generalization of uniformly hyperbolic (Anosov) systems, capturing many of their key features such as stability and rich ergodic behavior. 
Recall that a diffeomorphism $f \colon M \to M$ is called~\emph{partially hyperbolic} if the tangent bundle admits a continuous $Df$-invariant splitting
\[
TM = E^s \oplus E^c \oplus E^u,
\]
where $E^s$ is uniformly contracting, $E^u$ is uniformly expanding, and both dominate the center bundle $E^c$. More precisely, there exists an integer $k \geq 1$ such that for every $x \in M$,
\[
\|Df^k(x)|_{E^s}\| < 1 < \|(Df^k(x)|_{E^u})^{-1}\|^{-1},
\]
and
\[
\|Df^k(x)|_{E^s}\| < \|(Df^k(x)|_{E^c})^{-1}\|^{-1}, 
\qquad 
\|Df^k(x)|_{E^c}\| < \|(Df^k(x)|_{E^u})^{-1}\|^{-1}.
\]

Typical examples include time-one maps of Anosov flows and skew products of the form $f(x,y) = (g(x), g_x(y))$, where $g$ is an Anosov diffeomorphism and the contraction and expansion of the fiber maps $g_x$ are dominated by the hyperbolicity of the base dynamics.

It is well-known that the invariant distributions $E^s$ and $E^u$ are uniquely integrable into H\"older continuous foliations $\mathcal{W}^s$ and $\mathcal{W}^u$. The diffeomorphism $f$ is called~\textit{dynamically coherent} if, moreover, there exist $f$-invariant foliations $\mathcal{W}^{cs}$ and $\mathcal{W}^{cu}$ tangent to $E^{cs} := E^s \oplus E^c$ and $E^{cu} := E^c \oplus E^u$ respectively. We denote by $\mathcal{W}^c$ the center foliation (tangent to $E^c$) in that case. 

In this paper, we study the regularity of the extremal distribution $E^s \oplus E^u$. We show that the Lipschitz condition is highly rigid, leading to strong regularity and dynamical consequences.
\subsection{Extremal distributions and the Foulon--Hasselblatt cocycle}
It is well known (see, e.g.,~\cite{hps77,psw97}) that the bundles $E^s$ and $E^u$ (and hence $E^s \oplus E^u$) are always H\"older continuous. On the other hand, higher regularity such as Lipschitz or $C^1$ is expected to be exceptional and to impose strong restrictions on the dynamics. This phenomenon was established by Foulon--Hasselblatt~\cite{fh03} for $3$-dimensional conservative $C^r$ Anosov flows. They showed that once $E^s \oplus E^u$ is Lipschitz, then $E^s\oplus E^u$ is automatically $C^{r-1}$, and the flow is either contact or (constant roof) suspension. In higher dimensions, Ghys~\cite{ghys2006codimension} showed that for a codimension one $C^2$ Anosov flow, $C^1$ regularity of $E^s \oplus E^u$ forces the existence of a global cross section with constant return time.

A key tool in~\cite{fh03} is the so-called ``longitudinal KAM cocycle'' (which we refer to as the~\emph{Foulon--Hasselblatt cocycle} in the following), whose cohomology class vanishes if and only if $E^s \oplus E^u$ has the aforementioned higher regularity. Roughly speaking, this cocycle represents the first nontrivial derivative of the return times with respect to a suitable family of transversals to the flow. This construction relies essentially on the flow structure and does not admit a direct analogue for diffeomorphisms, where no notion of return time is available.

In~\cite{glh25}, Gogolev--Leguil--Rodriguez Hertz developed a related cocycle-based approach in the setting of $3$-dimensional dissipative Anosov flows. They showed that the polynomial normal forms introduced by Tsujii--Zhang~\cite{tz23} can be interpreted as defining a twisted cocycle (with twist given by the Jacobian), which measures the obstruction to uniformly diagonalizing the differential of the flow. In the conservative case, this cocycle reduces to the Foulon--Hasselblatt cocycle, corresponding to the linear term of the polynomial normal form.

Motivated by this perspective, we define an analogue of the Foulon--Hasselblatt cocycle using polynomial normal forms in the spirit of Tsujii--Zhang~\cite{tz23} and Eskin--Potrie--Zhang~\cite{epz} (see Section~\ref{sec_FH_cocycle}). This construction does not rely on return times and applies to general $3$-dimensional partially hyperbolic systems. This allows us to extend the rigidity phenomenon of Foulon--Hasselblatt~\cite{fh03} to the partially hyperbolic setting. Our main result is as follows: 
\begin{mainthm}\label{theorem main}
 Let $f\colon M \to M$ be a $C^\infty$ volume-preserving partially hyperbolic diffeomorphism on a closed 3-manifold $M$. If $E^s \oplus E^u$ is Lipschitz, then 
 \begin{enumerate}
     \item \label{t1 dc}$f$ is dynamically coherent;
     \item  \label{t1 boot} the distribution $E^s\oplus E^u$ is $C^\infty$, and the distributions $E^s$, $E^u$ are $C^{1+\alpha}$ for some $\alpha>0$;
     \item  \label{t1 lambda c}the center cocycle $x \mapsto Df(x)|_{E^c}$ is cohomologous to a constant $\lambda^c$. Besides, $\lambda^c \neq \pm1$ if and only if $f$ is an Anosov diffeomorphism for which $E^s \oplus E^u$ is integrable. 
 \end{enumerate}

Moreover, we have the following dichotomy:
 \begin{itemize}
     \item either $E^s \oplus E^u$ is integrable;
     \item or $f$ is accessible and preserves a $C^\infty$ contact form, and $E^c$ is $C^\infty$. Up to a finite cover and iteration, $f$ is $C^\infty$-conjugate to one of the following:
\begin{itemize}
    \item an isometric extension of a volume-preserving Anosov diffeomorphism on $\TT^2$ with total space $M$ a Heisenberg nilmanifold;
    \item the time-one map of a contact Anosov flow.
    \end{itemize}
      \end{itemize}
\end{mainthm}\begin{remark}\label{remark 1.1}
Theorem~\ref{theorem main} is~\textit{sharp} in the following sense:
\begin{itemize}
    \item The Lipschitz assumption on $E^s \oplus E^u$ is essentially optimal; even if $E^s \oplus E^u$ is H\"older continuous for every exponent $\alpha \in (0,1)$, the conclusion need not hold (see Example~\ref{example su}).
    \item The foliation $\W^c$ may fail to be absolutely continuous when $E^s \oplus E^u$ is integrable; see Example~\ref{example c}.
  \item In the classification of the accessible case, one cannot in general expect further rigidity in either of the two cases; indeed, such diffeomorphisms need not even be $C^0$-conjugate to algebraic models (see Section~\ref{sec_example} and~\cite{fhflow}). To obtain further rigidity, one must assume in addition that $E^s$ and $E^u$ are $C^{1+\mathrm{Lip}}$, in which case $f$ is $C^\infty$-conjugate to an algebraic model (see Corollary~\ref{theorem c2}).
\end{itemize}

Moreover, Theorem~\ref{theorem main} admits a finite-regularity counterpart: if $f$ is $C^r$ ($r \geq 5$) and $E^s \oplus E^u$ is Lipschitz, then $E^s \oplus E^u$ is $C^{r-1}$; see Remark~\ref{finite_reg_version}.
\end{remark}


\subsection{Applications to the rigidity of $u$-Gibbs states}
Recently, the lack of regularity of $E^s$ and $E^u$ has been linked to the study of $u$-Gibbs states for $3$-dimensional partially hyperbolic diffeomorphisms with expanding center. Roughly speaking, a $u$-Gibbs state is an invariant Borel probability measure whose conditional measures along strong unstable manifolds are absolutely continuous. We focus on the case where the center Lyapunov exponent is positive.

Building on ideas from homogeneous dynamics, Teichm\"uller dynamics, and random dynamics~\cite{bq11,bh17,em18,el18}, Katz~\cite{kat23} proved a rigidity result for such measures in the $C^\infty$ setting\footnote{The smoothness assumption is essential; for an alternative approach under weaker regularity, see~\cite{alos24}.}: if a $u$-Gibbs state satisfies the quantitative non-integrability (QNI) condition, then it is in fact an SRB measure, i.e., its conditional measures along the full unstable Pesin manifolds are absolutely continuous.

A key difficulty in applying this criterion is to express the QNI condition in geometric terms. This was achieved by Eskin--Potrie--Zhang~\cite{epz}, who showed that the failure of $\ell$-integrability (see Definition~\ref{def JI l}) implies the QNI condition. In particular, given a partially hyperbolic $u$-Gibbs state $\mu$ such that $E^s \oplus E^u$ is not $C^1$ on its support, then the support of $\mu$ is not $\ell$-integrable for any $\ell \ge 2$ (see Lemma~\ref{lemma lip JI l}) and hence $\mu$ satisfies QNI (see~\cite[Theorem 8.1]{epz}). 

This naturally raises the question of when $\ell$-integrability is equivalent to the integrability of $E^s \oplus E^u$. In a forthcoming work by Avila--Crovisier--Eskin--Potrie--Wilkinson--Zhang~\cite{acepwz}, it is shown that for partially hyperbolic Anosov diffeomorphisms on $\TT^3$, $\ell$-integrability for sufficiently large $\ell$ implies the integrability of $E^s \oplus E^u$, which in particular leads to the uniqueness of the $u$-Gibbs state when $E^s \oplus E^u$ is not integrable. The uniqueness of $u$-Gibbs states in the case where $E^s \oplus E^u$ is integrable is also investigated in a joint work~\cite{aclos} by Alvarez--Crovisier--Leguil--Obata--Santiago. 


Since $\ell$-integrability for $\ell \ge2$ implies that $E^s \oplus E^u$ is $C^1$.
In the conservative setting, Theorem~\ref{theorem main} yields: 
\begin{coralph}\label{theorem QNI}
Let $f\colon M \to M$ be a $C^\infty$ volume-preserving partially hyperbolic diffeomorphism on a closed $3$-manifold. Then $f$ is $\ell$-integrable for some $\ell > 2$ if and only if $E^s \oplus E^u$ is integrable.

Moreover, assume that $E^s \oplus E^u$ is not globally jointly integrable. Let $A\subset M$ be an accessibility class, and let $\Lambda:=\overline{A}$. Then, up to passing to a finite iterate and finite covers, for some integer
 $n \geq 1$, exactly one of the following holds: 
\begin{enumerate}
    \item $E^s \oplus E^u|_\Lambda$ is integrable, and $A$ is an $f^n$-invariant $2$-torus tangent to $E^s\oplus E^u$;
    \item\label{point_deux_corb} $A$ is an $f^n$-invariant open accessibility class, and $\Lambda$ is not $\ell$-integrable for any $\ell > 2$.
\end{enumerate}
\end{coralph}

As a consequence, combining our results with those of~\cite{epz}, we obtain the following rigidity theorem for $u$-Gibbs states.
\begin{coralph}\label{theorem measure rigid}
Let $f\colon M \to M$ be a $C^\infty$ partially hyperbolic diffeomorphism on a closed $3$-manifold that preserves a probability volume measure $m$. Then,
\begin{itemize}
    \item \begin{enumerate}
\item either $E^s\oplus E^u$ is integrable;
\item\label{point_deux_dichot} or there exists $n \geq 1$ such that any ergodic $u$-Gibbs state $\mu$ with a positive center Lyapunov exponent and $m(\mathrm{supp}(\mu))>0$ coincides with the restriction $m_U:=\frac{m|_U}{m(U)}$ of $m$ to $U:=\cup_{k=0}^{n-1} f^k(U_0)$, where $U_0\subset M$ is an $f^n$-invariant open accessibility class (and $\mu$ satisfies the QNI condition);  
\end{enumerate}
\item If,  in addition,  the strong unstable foliation $\mathcal{W}^u$ tangent to \(E^u\) is minimal, then in case~\eqref{point_deux_dichot} above, the condition \(m(\operatorname{supp}(\mu))>0\) can be omitted, and hence \(\mu=m\). 

\noindent In particular, for a \(C^1\)-open and dense subset of $3$-dimensional volume-preserving partially hyperbolic diffeomorphisms, the invariant volume $m$ is the unique ergodic $u$-Gibbs state (or $s$-Gibbs state) with a nonzero center Lyapunov exponent. 
\end{itemize}
\end{coralph}

\begin{remark}For a general volume-preserving partially hyperbolic diffeomorphism on a closed $3$-manifold, the condition \( m(\mathrm{supp}(\mu)) > 0 \) in case~\eqref{point_deux_dichot} of Corollary~\ref{theorem measure rigid} cannot be omitted. In fact, even when $E^s\oplus E^u$ is not jointly integrable, certain accessibility classes may be contained in  $2$-tori tangent to $E^s\oplus E^u$ (see, e.g., Proposition~\ref{prop ham17hp25} for further details). This prevents both accessibility and the minimality of the strong unstable foliation. In particular, for such systems, without assuming \( m(\mathrm{supp}(\mu)) > 0 \), the \( u \)-Gibbs measure \( \mu \) might be supported on one of these \( su \)-tori, and \( \mu \) could fail to be a physical measure for \( f \).  

On the other hand, for a $C^1$-open and dense subset of the diffeomorphisms considered in Corollary~\ref{theorem measure rigid}, the strong unstable foliation is minimal. This follows from a combination of two recent works of Avila--Crovisier--Wilkinson~\cite{avila2025minimalitystrongfoliationsanosov} and Crovisier--Potrie~\cite{CP}; see Remark~\ref{remark minimal} for more details.

\end{remark}

\subsection{Applications to the classification of 3D partially hyperbolic diffeomorphisms}
In 2001, Pujals proposed a conjecture to address the classification problem for partially hyperbolic diffeomorphisms on 3-manifolds up to leaf conjugacy which was formulated in~\cite{bw05} (see also~\cite{crru18} for a modified version of this conjecture). Both conjectures turned out to be false (see~\cite{bgp16,bpp16,bghp20}), but it was later understood that the counterexamples still had some connection with Anosov flows (see~\cite[Theorem~A]{bfp23}). This allowed for a reformulation of the conjecture which involved introducing the notion of collapsed Anosov flows, see for instance~\cite[Question~1]{bfp23}, and was recently proved by Fenley--Potrie~\cite{fp25}.

In general, one cannot expect a classification of (partially) hyperbolic diffeomorphisms up to $C^1$ (or $C^\infty$) conjugacy, as the stable and unstable distributions may be merely Hölder continuous. This lack of regularity is a fundamental obstruction to smooth rigidity. On the other hand, when the invariant distributions of a hyperbolic diffeomorphism exhibit higher ($C^k$, $k\geq 2$) regularity, it often forces strong rigidity, sometimes implying that the system is $C^\infty$-conjugate to an algebraic model. This phenomenon was first observed by A.~Avez~\cite{ave68} for Anosov diffeomorphisms on $\TT^2$ with $C^\infty$ distributions, and was further developed for general Anosov diffeomorphisms and flows (see~\cite{hk90,ghy87,ghy93,bl93,bfl92}).

It is therefore natural to ask whether a similar rigidity phenomenon persists for partially hyperbolic systems. However, in higher dimensions this question becomes substantially more subtle, as the topological types of partially hyperbolic diffeomorphisms can be very complicated~\cite{goh15}, and it is not even clear what the appropriate algebraic models should be. In dimension $3$, motivated by the Pujals conjecture and the search for smooth rigidity, Carrasco--Pujals--Hertz~\cite{cph21} posed the following question:

\begin{question}[{see~\cite[Question 2]{cph21}}]\label{que cph}
Can one obtain a $C^\infty$ classification for partially hyperbolic diffeomorphisms on $3$-manifolds when $E^s$, $E^c$, and $E^u$ are $C^\infty$?
\end{question}

Under additional assumptions, several positive results have been obtained. When the invariant distributions are $C^\infty$ and the pointwise Lyapunov exponents are constant, Carrasco--Pujals--Hertz~\cite{cph21} obtain a $C^\infty$ classification. It was later observed by Allout--Moghaddamfar~\cite{am23} that this classification remains valid in the $C^1$ setting. Mion-Mouton~\cite{mmm22} also obtains a $C^\infty$ classification when the invariant distributions are $C^\infty$ and $E^s \oplus E^u$ carries a contact structure.

By combining Theorem~\ref{theorem main} with results of Avila--Viana--Wilkinson~\cite{avw14,avw22}, we obtain the following result, showing that in the volume-preserving setting, even a surprisingly weak regularity assumption is enough to force full smooth rigidity.

\begin{coralph}\label{theorem smooth}
Let $f \colon M \to M$ be a $C^\infty$ volume-preserving partially hyperbolic diffeomorphism on a closed $3$-manifold. If $E^s \oplus E^u$ is Lipschitz and $\W^c$ is absolutely continuous, then both $E^s \oplus E^u$ and $E^c$ are $C^\infty$. Up to a finite cover and an iterate, $f$ is $C^\infty$-conjugate to one of the following:
\begin{enumerate}
    \item an Anosov automorphism of $\mathbb{T}^3$;
    \item an isometric extension of a volume-preserving Anosov diffeomorphism on $\TT^2$;
    \item the time-one map of a contact or suspension Anosov flow.
\end{enumerate}
\end{coralph}
\begin{remark}
Note that, by Theorem~\ref{theorem main}, the Lipschitz regularity of $E^s \oplus E^u$ already implies dynamical coherence, so the additional assumption is only the absolute continuity of $\W^c$. Moreover, the regularity assumptions above are essentially optimal; see Remark~\ref{remark 1.1}, Examples~\ref{example su} and~\ref{example c}.
\end{remark}
If one further assumes that $E^s$ and $E^u$ are $C^{1+\mathrm{Lip}}$, i.e., $C^1$ with Lipschitz continuous derivatives, then a stronger classification can be obtained.\begin{coralph}\label{theorem c2}
Let $f \colon M \to M$ be a $C^\infty$ volume-preserving partially hyperbolic diffeomorphism on a closed $3$-manifold. Assume that $E^s$ and $E^u$ are $C^{1+\mathrm{Lip}}$ and $\W^c$ is absolutely continuous; then $E^s$, $E^u$, and $E^c$ are all $C^\infty$. Up to a finite cover and an iterate $f$ is $C^\infty$-conjugate to one of the following:
\begin{enumerate}
    \item an affine automorphism of a nilmanifold;
    \item a constant roof suspension of a hyperbolic automorphism of $\TT^2$;
    \item \label{hyperbolic surface}the time-one map of a smooth time-change of the geodesic flow on a closed hyperbolic surface.
\end{enumerate}
\end{coralph}\begin{remark}
The assumption that $E^s$ and $E^u$ are $C^{1+\mathrm{Lip}}$ is essentially optimal: for each $\alpha \in (0,1)$, there exist examples for which $E^s$ and $E^u$ are $C^{1+\alpha}$ but the conclusion of Corollary~\ref{theorem c2} fails (see Lemma~\ref{lemma c2-}). Moreover, the time-change in Case~\eqref{hyperbolic surface} is not arbitrary: it is a special algebraic time-change coming from a morphism from a lattice to the diagonal subgroup; see~\cite{ghy87} for a precise description.
\end{remark}We also have a classification for contact partially hyperbolic diffeomorphisms on $3$-manifolds. Recall that a $C^1$ $1$-form $\alpha$ is called~\textit{contact} if $\alpha \wedge d\alpha$ is non-vanishing everywhere. 
\begin{coralph}\label{theorem contact}
Let $f \colon M \to M$ be a $C^\infty$ partially hyperbolic diffeomorphism on a closed $3$-manifold preserving a $C^{1}$ contact form $\alpha$: $f^* \alpha=\alpha$.  Then $\alpha$ is $C^\infty$, $f$ is volume-preserving, and the distributions $E^s \oplus E^u$ and $E^c$ are $C^\infty$. Up to a finite cover and an iterate, $f$ is $C^\infty$-conjugate to one of the following:
\begin{itemize}
    \item an isometric extension of a volume-preserving Anosov diffeomorphism on $\TT^2$ with total space $M$ a Heisenberg nilmanifold;
    \item the time-one map of a contact Anosov flow.
\end{itemize}
\end{coralph}
\subsection{Further remarks and examples}\label{sec remark}
Point~\eqref{t1 lambda c} in Theorem~\ref{theorem main} is related to a result of Gan--Shi~\cite[Theorem~1.1]{gs20}. They show that for a $C^{1+\alpha}$ ($\alpha > 0$) volume-preserving partially hyperbolic diffeomorphism on $\mathbb{T}^3$ homotopic to an Anosov automorphism, if $E^s \oplus E^u$ is integrable (and hence Lipschitz; see e.g.~\cite{ks25}), then the cocycle $x \mapsto Df(x)|_{E^c}$ is cohomologous to a constant $\lambda^c$.
Our result extends this phenomenon to a broader setting: assuming only that $E^s \oplus E^u$ is Lipschitz, without any integrability or topological assumptions, we obtain the same cohomological rigidity. This suggests that the rigidity of the center dynamics can already be detected at the level of the regularity of $E^s \oplus E^u$.

There exist accessible volume-preserving partially hyperbolic diffeomorphisms on $3$-manifolds with non-solvable fundamental group that are not homotopic to the identity (see~\cite{bgp16,bghp20}). In particular, Theorem~\ref{theorem main} implies that the extremal distribution $E^s \oplus E^u$ of such diffeomorphisms is never Lipschitz. Moreover, any volume-preserving partially hyperbolic diffeomorphism homotopic to these examples also fails to admit a Lipschitz extremal distribution. This stands in contrast with the Anosov setting, where every Anosov diffeomorphism on a $3$-manifold is homotopic to an affine automorphism of $\TT^3$.

In a somewhat dual direction, several works study rigidity phenomena arising from the regularity of the center foliation. Avila--Viana--Wilkinson~\cite{avw14} showed that for perturbations of the time-one map of a geodesic flow on negatively curved surfaces, absolute continuity of the center foliation forces the system to embed into a smooth flow, while otherwise the disintegration is atomic. Related rigidity results have since been obtained in various settings; see for instance,~\cite{gog12,avw22,gks23}. Our results can be viewed as complementary to this line of work, replacing the role of $E^c$ by that of $E^s \oplus E^u$. For instance, Theorem~\ref{theorem main} implies that for any volume-preserving partially hyperbolic diffeomorphism homotopic to the time-one map of a geodesic flow on a negatively curved surface, either $E^s \oplus E^u$ fails to be Lipschitz, or the system embeds as the time-one map of a contact flow.

In our setting, we also obtain smoothness of the center bundle $E^c$ (when $f$ is accessible in Theorem~\ref{theorem main}), but via a different mechanism: we exploit the smoothness of $E^s \oplus E^u$ together with the associated contact structure (see Corollary~\ref{coro ec smooth}), which in turn yields the smoothness of $E^c$. 

It is then natural to ask whether a similar phenomenon can occur when $E^s \oplus E^u$ is integrable and the individual distributions have higher regularity. Note that the assumption that $E^s \oplus E^u$ is $C^\infty$ alone is not sufficient; see Example~\ref{example c}.

\begin{question}
Let $f\colon M \to M$ be a volume-preserving partially hyperbolic diffeomorphism on a $3$-manifold $M$. If $E^s \oplus E^u$ is integrable and both $E^s$ and $E^u$ are $C^2$, is $E^c$ absolutely continuous?
\end{question}

We conclude with several examples highlighting the sharpness of our results, showing in particular that both the Lipschitz regularity of $E^s \oplus E^u$ is indispensable for Theorem~\ref{theorem main}, and the absolute continuity of $\W^c$ are indispensable for Corollary~\ref{theorem smooth}. 
\begin{example}\label{example su}
Even if $E^s \oplus E^u$ is $\theta$-H\"older for every $\theta \in (0,1)$, it need not be Lipschitz. It is shown in~\cite{fh03} that for every volume-preserving Anosov flow $\varphi_t$ on a closed $3$-manifold, the distribution $E^s \oplus E^u$ is $\theta$-H\"older for all $\theta \in (0,1)$. However, $E^s \oplus E^u$ fails to be Lipschitz unless $\varphi_t$ is a contact or suspension Anosov flow. 

We further remark that this threshold can be improved when the center bundle $E^c$ is uniformly expanding or contracting; in that case, a sufficiently high Hölder regularity of $E^s \oplus E^u$ already forces smoothness. More precisely, Xu--Zhang~\cite{xz24} proved that for diffeomorphisms $C^1$-close to a volume-preserving partially hyperbolic system, if the Hölder exponent of $E^s \oplus E^u$ exceeds a certain critical value $\theta_0 \in (0,1)$, then $E^s \oplus E^u$ is automatically $C^\infty$.
\end{example}
\begin{example}\label{example c}
Even when $E^s \oplus E^u$ is integrable, the center foliation $\W^c$ need not be absolutely continuous. Two typical constructions are as follows:
\begin{itemize}
    \item \emph{Katok's example:} Let $f = A \times \mathrm{Id}$ be the product of an Anosov automorphism $A$ on $\TT^2$ and the identity on $\TT^1$. One can perturb $f$ along each $su$-torus so that the center holonomy fails to be continuous (see~\cite{mil97} for details), which implies that $\W^c$ is not absolutely continuous.
    
    \item \emph{Gogolev's example:} In~\cite{gog12}, Gogolev showed that for volume-preserving Anosov diffeomorphisms on $\TT^3$, the strong unstable Jacobian is a coboundary if and only if the center foliation $\W^c$ is absolutely continuous. Therefore, by perturbing the Jacobian to destroy the coboundary property, one obtains examples where $\W^c$ is not absolutely continuous.
\end{itemize}
\end{example}


\begin{example}\label{example single}
If only $E^s$ or $E^u$ is $C^1$, then $E^s \oplus E^u$ need not be Lipschitz. For instance, for a volume-preserving partially hyperbolic diffeomorphism on $\TT^3$ with uniformly expanding center bundle $E^c$, the stable bundle $E^s$ is always $C^1$. However, the unstable bundle $E^u$ may fail to be Lipschitz and can even exhibit fractal behavior~\cite{xz24} when the system is accessible.
\end{example}


\subsection{Structure of the paper}
This paper is organized as follows.
In Section~\ref{sec_prelim}, we collect the preliminary material used throughout the paper. 
In Section~\ref{sec acc and iso}, we study the $1$-form associated with the Lipschitz distribution $E^s \oplus E^u$ and show that the cocycle $Df|_{E^c}$ is a coboundary on every open accessibility class. 
In Section~\ref{sec_FH_cocycle}, we recall the notions of adapted charts and templates and introduce the Foulon--Hasselblatt cocycle for general $3$-dimensional partially hyperbolic diffeomorphisms. This cocycle allows us to bootstrap the Lipschitz regularity to $C^1$ regularity, and, under $\infty$-bunching, to $C^\infty$ regularity on every open accessibility class. 
In Section~\ref{sec_dicho}, we establish a dichotomy between accessibility and integrability. 
In Section~\ref{sec_proofs}, we prove the main theorem and its corollaries. 
Finally, in Section~\ref{sec_example}, we study a specific example on a Heisenberg manifold.

\subsection*{Acknowledgments.} We are grateful to Rafael Potrie for a very careful reading of an earlier version of this manuscript and for numerous helpful comments and suggestions that improved the paper.
We are grateful to Amie Wilkinson for many helpful discussions and for posing a question that motivated Example~\ref{example contact}.
We are also grateful to Christian Bonatti, Sylvain Crovisier, Danijela Damjanovic, Ziqiang Feng, Shaobo Gan, Andrey Gogolev, Carlos Matheus, Federico Rodriguez Hertz, Yi Shi, Masato Tsujii, Danyu Zhang, and Zhiyuan Zhang for many helpful discussions.

\section{Preliminaries}\label{sec_prelim}
\subsection{Contact form and Reeb vector field}
Let $M$ be a closed $3$-manifold and let $E\subset TM$ be a $2$-dimensional $C^r$ distribution 
($r\in[0,\infty]$ or $r=\mathrm{Lip}$). 
\begin{lemma}\label{lemma alpha}
If $E$ is co-orientable, i.e.\ there exists a continuous $1$-dimensional distribution $F$ transverse to $E$, then there exists a $C^r$ nowhere-vanishing $1$-form $\alpha$\footnote{The choice of $\alpha$ is not unique: $\varphi\alpha$ is also a 1-form satisfying the condition for any nowhere vanishing $C^r$ function $\varphi$.} such that $\ker\alpha=E$.
\end{lemma}
\begin{proof}
Since $E$ is co-orientable, the normal bundle $TM/E$ is an orientable line bundle and thus a trivial bundle. 
Identifying the annihilator bundle
\[
E^\circ=\{\beta\in T^*M:\beta|_E=0\}\cong (TM/E)^*,
\]
we obtain that $E^\circ$ is a trivial $C^r$ line subbundle of $T^*M$. Therefore it admits a global $C^r$ nowhere-vanishing section $\alpha$, i.e., a $C^r$ 1-form with $\ker \alpha= E$. 
\end{proof}

In particular, if $E$ is a Lipschitz distribution, the 1-form $\alpha$ is Lipschitz, and is differentiable almost everywhere. The differential $d\alpha$ is a well-defined measurable ($L^\infty$) 2-form. The Frobenius theorem about integrablity of distributions also holds for Lipschitz distributions\cite[Theorem~3]{fh03}. 
\begin{proposition}[{see~\cite[Corollary A.3]{lipfro}}]\label{prop frobenius}
  If $E$ is Lipschitz, then $E$ is uniquely integrable if and only if $\theta \equiv 0$, where $\theta := \alpha \wedge d\alpha$. 
\end{proposition}

If $\alpha$ is $C^k$ for some $k \geq 1$ and $\theta$ is non-vanishing everywhere, then $\alpha$ is called a~\textit{contact form}. To every contact form, one can uniquely associate its~\textit{Reeb vector field} (see, for example,~\cite{bookcontact} for more details).
\begin{proposition}\label{prop frobinius}
    Let $\alpha$ be a $C^k$ contact form. Then there is a uniquely defined $C^{k-1}$ vector field $R = R_\alpha$ --- called the~\textit{Reeb vector field} --- such that
\[
\alpha(R) = 1, \qquad \iota_R d\alpha := d\alpha(R, \cdot) = 0.
\] 
\end{proposition}
We also recall an elementary fact about contact forms. 
\begin{lemma}\label{lemma non dege}
Let $\beta$ be a $C^1$ contact form on a $3$-manifold $M$ and $R$ its Reeb vector field.  
Let $S \subset M$ be a $2$-dimensional $C^1$ submanifold transverse to $R$. Then $d\beta|_{TS}$ is nondegenerate. 
\end{lemma}
\begin{proof}
For any vector fields $X,Y$ taking values in $TS$, using \(\iota_R d\beta= 0\) and \(\beta(R) = 1\), we have
\[
(\beta \wedge d\beta)(X, Y, R)
= \beta(X) d\beta(Y, R) - \beta(Y) d\beta(X, R) + \beta(R) d\beta(X, Y)
= d\beta(X, Y).
\]
Since \(\beta \wedge d\beta\) is a volume form, we have \(d\beta(X, Y) \neq 0\), and thus $d\beta|_{TS}$ is nondegenerate. 
\end{proof}

\subsection{Oseledets theorem and Lyapunov exponents}
Let $X$ be a compact metric space and $E \to X$ be a (continuous) vector bundle. Then for any homeomorphism $f \colon X \to X$, a continuous linear~\emph{cocycle} over $T$ is a bundle map
$F \colon E \to E$ covering $T$. 

\begin{proposition}[Oseledets theorem]
    Suppose that $f$ preserves an ergodic probability measure $\mu$ on $X$, and $E$ is equipped with a continuous Riemannian metric. Then there are real numbers
\[\chi_1<\chi_2<\cdots<\chi_k\]
and a measurable $F$-invariant splitting 
\[
E = E^{\chi_1} \oplus \cdots \oplus E^{\chi_k},
\]
such that for $v \in E_x \setminus \{0\}$,
\[
v \in E_x^{\chi_i}
\quad \Longleftrightarrow \quad
\lim_{n \to \pm \infty} \frac{1}{n}
\log \bigl\| F^n(v) \bigr\| = \chi_i .
\]
\end{proposition}
The splitting $E = \bigoplus_i E^{\chi_i}$ is called the~\emph{Oseledets splitting} for the cocycle, and the numbers $\chi_i$ are called~\textit{Lyapunov exponents}. 

The following well-known result allows one to deduce~\emph{uniform growth} of cocycles from knowledge about exponents for~\emph{every} invariant measure. The proof is a corollary of a classical result on subadditive sequences (see~\cite{sch98} or~\cite[Chapter 4]{kal11}).  

\begin{lemma}[{see e.g.~\cite[Lemma 10]{dwx21}}]\label{lemma dwx}
Let $f \colon X \to X$ be a continuous map of a compact  metric space, and let $F \colon E \to E$ be a continuous linear cocycle over $f$. 
\begin{enumerate}
\item If for any $f$-invariant ergodic measure $\nu$, the maximal Lyapunov exponent 
$\chi^{\max}(F,\nu)$ satisfies
$\chi^{\max}(F,\nu) \le \chi$, then for any $\varepsilon > 0$, there exists
$n \in \mathbb{Z}^+$ such that
\[
\|F^n(x)\| \le e^{n(\chi+\varepsilon)}, \quad \forall\, x \in X.
\]
\item If for any $f$-invariant ergodic measure $\nu$, the minimal Lyapunov exponent $\chi^{\min}(F,\nu)$ satisfies
$\chi^{\min}(F,\nu) \ge \chi'$, then for any $\varepsilon > 0$, there exists
$n \in \mathbb{Z}^+$ such that
\[
\bigl\|F^n(x)^{-1}\bigr\|^{-1} \ge e^{n(\chi' - \varepsilon)}, \quad \forall\, x \in X.
\]
\end{enumerate}
\end{lemma}

\subsection{Partially hyperbolic systems on 3-manifolds}
Let $f \colon M \to M$ be a $C^1$ diffeomorphism of a closed 3-manifold $M$. We say that $f$ is~\textit{partially hyperbolic} if there exists a continuous $Df$-invariant splitting of the tangent bundle $TM = E^s \oplus E^c \oplus E^u$ such that
\[
\|Df^k(x)|_{E^s}\|
<
\min\bigl\{\|Df^k(x)|_{E^c}\|, 1\bigr\}
\le
\max\bigl\{\|Df^k(x)|_{E^c}\|, 1\bigr\}
<
\|Df^k(x)|_{E^u}\|
\]
for some $k \in \mathbb{Z}^+$ and every $x \in M$.

We denote by $\mathcal{W}^s$ and $\mathcal{W}^u$ the stable and unstable foliations, tangent respectively to $E^s$ and $E^u$; recall that they are always H\"older continuous. When $f$ is  dynamically coherent, we denote by $\mathcal{W}^{cs}$ and $\mathcal{W}^{cu}$ the foliations tangent to $E^{cs} := E^s \oplus E^c$ and $E^{cu} := E^c \oplus E^u$ respectively, and we denote by $\mathcal{W}^c$ the center foliation tangent to $E^c$. 
\medskip
\paragraph{\textbf{Accessibility classes}}
The foliations $\mathcal{W}^s$ and $\mathcal{W}^u$ induce an equivalence relation on $M$: we say that $x, y \in M$ are in the same~\textit{accessibility class} if they can be joined by an $su$-path, that is, a piecewise $C^1$ path such that each piece is contained in a single $\mathcal{W}^s$-leaf or a single $\mathcal{W}^u$-leaf.
The diffeomorphism $f$ is called~\textit{accessible} if the whole manifold $M$ itself is an accessibility class.

Recently, Fenley--Potrie~\cite{fp25} obtained a full classification of volume-preserving partially hyperbolic diffeomorphisms on a 3-manifold $M$ up to leaf conjugacy. In particular, their results imply that every $C^2$ volume-preserving partially hyperbolic diffeomorphism on a 3-manifold with non-solvable fundamental group is accessible. The accessibility classes in the case where $M$ has solvable fundamental group were studied by Hammerlindl~\cite{ham17}. Combining these results, we obtain the following classification of accessibility classes.

\begin{proposition}[{see~\cite[Theorem~3.4]{ham17},~\cite[Corollary~1.1]{fp25}}]\label{prop ham17hp25}
Let $f \colon M \to M$ be a $C^2$ volume-preserving partially hyperbolic diffeomorphism on a $3$-manifold. Up to considering a finite iterate and finite covers, one of the following occurs:
\begin{enumerate}
    \item\label{cas_accessib} $f$ is accessible;
    \item $E^s \oplus E^u$ is integrable and $f$ is topologically conjugate to one of the following:
    \begin{itemize}
        \item an affine automorphism on $\mathbb T^3$;
        \item the time-one map of a suspension Anosov flow;
    \end{itemize}
    \item\label{cas_trois} there exist $n \geq 1$, a $C^1$ surjection $p \colon M \to \mathbb{S}^1$, and an open set $\emptyset\neq V \subset \mathbb{S}^1$ such that
    \begin{itemize}
        \item for every connected component $I$ of $V$, the set $p^{-1}(I)$ is an $f^n$-invariant subset homeomorphic to $\mathbb{T}^2 \times I$, and the restriction $f^n|_{p^{-1}(I)}$ is accessible and ergodic;
        \item for every $t \in \mathbb{S}^1 \setminus V$, the fiber $p^{-1}(t)$ is an $f^n$-invariant $2$-torus tangent to $E^s \oplus E^u$.
    \end{itemize}
\end{enumerate}
\end{proposition}
We record here some more properties about the non-accessible cases. 
\begin{lemma}[{see~\cite[Theorem~2.4]{ham17}}]\label{lemma ham measure}
Let $f \colon M \to M$ be a diffeomorphism as in case~\eqref{cas_trois} of Proposition~\ref{prop ham17hp25}. Then any $f^n$-invariant ergodic measure is either supported on an $f^n$-invariant $2$-torus or on the closure of an $f^n$-invariant open accessibility class. 
\end{lemma}

\begin{lemma}[{see~\cite[Theorem 5.1]{ham17}}]\label{lemma ab dc}
    Let $f \colon M \to M$ be a diffeomorphism as in case~\eqref{cas_trois} of Proposition~\ref{prop ham17hp25}; then $f$ is dynamically coherent. 
\end{lemma}

\begin{lemma}\label{lemma ks06}
Let $f \colon M \to M$ be a $C^\infty$ partially hyperbolic diffeomorphism on a 3-manifold $M$. If $S\subset M$ is a $C^1$ submanifold tangent to $E^s \oplus E^u$, then $S$ is a $C^\infty$ submanifold of $M$. Moreover, if $E^s\oplus E^u$ is integrable to a continuous foliation $\W^{su}$, then the  leaves of $\W^{su}$ are uniformly $C^\infty$. 
\end{lemma}
\begin{proof}
The lemma is  proved by Kalinin--Sadovskaya in a more general form~\cite[Lemma 4.1]{ks06} via a beautiful  application of Journé's regularity lemma~\cite{jou88}.
\end{proof}

\paragraph{\textbf{Bunching and strong bunching}}
We then recall the bunching and strong bunching conditions for partially hyperbolic diffeomorphisms. We only focus on the 3-dimensional case for convenience (see~\cite{wil13} for the general case). 

Let $f \colon M \to M$ be a $C^\infty$ partially hyperbolic diffeomorphism of a 3-manifold $M$. For any $x \in M$ and $* = s,c,u$, denote
\[\lambda^\ast_x(n):=\|D f^n(x)|_{E^\ast}\|,\quad \forall\, n \in \ZZ.\]
We also abbreviate $\lambda^\ast_x:=\lambda^\ast_x(1)$. 
\begin{definition}
For any given $r \geq 1$, we say that $f$ is~\emph{$r$-bunched} if
\[
\lambda_x^s(k) < \bigl(\lambda_x^c(k)\bigr)^r < \lambda_x^u(k),
\quad
\lambda_x^s(k) < \bigl(\lambda_x^c(k)\bigr)^{1-r} < \lambda_x^u(k),\quad \forall\, x \in M
\]
for some $k \in \mathbb{Z}^+$. We say that $f$ is~\emph{strongly $r$-bunched} on if $f$ is $r$-bunched and 
\[\lambda_x^s(k) < \bigl(\lambda_x^c(k)\bigr)^{-r} < \lambda_x^u(k)\quad \forall\, x \in M\]
for some $k \in \mathbb{Z}^+$. Moreover, $f$ is (strongly) $\infty$-bunched if $f$ is (strongly) $r$-bunched for every $r \geq 1$. 
\end{definition}
It is direct to check that if $f$ is volume-preserving, then $f$ is strongly $1$-bunched.

\begin{proposition}[see~\cite{hps77,psw97}]\label{prop r-bunched}
If $f$ is dynamically coherent and $r$-bunched, then
    \begin{enumerate}
        \item the leaves of $\mathcal{W}^{cs}$, $\mathcal{W}^{cu}$, and $\mathcal{W}^c$ are $C^r$; 
        \item the distributions $E^s$ and $E^u$ are $C^{r-\epsilon}$ along $\mathcal{W}^c$ for any $\epsilon>0$. 
    \end{enumerate}
\end{proposition}
\subsection{Cohomological equation and Livshits theorem}
Recall that for a dynamical system $f \colon M \to M$ and a function $\phi \colon M \to \mathbb{R}$, if there exists a function $\Phi\colon M \to \mathbb{R}$ such that
\begin{equation}\label{eq:cohomological}
\phi = \Phi \circ f - \Phi,
\end{equation}
then $\Phi$ is called a solution of the~\textit{cohomological equation}~\eqref{eq:cohomological}, and $\phi$ is called a coboundary.

We then present a Livshits type theorem for a partially hyperbolic diffeomorphism on a 3-manifold, which is a special case of~\cite[Theorem A]{wil13}. 
\begin{proposition}\label{prop wil13}
Let $f \colon M \to M$ be a partially hyperbolic diffeomorphism on a 3-manifold. Let $U \subset M$ be an open accessibility class such that $f|_U$ is bijective and ergodic.
\begin{enumerate}
    \item If $f$ is $C^2$ and volume-preserving, $\phi$ is H\"older continuous, and there exists a measurable solution $\Phi$ of equation~\eqref{eq:cohomological}, then there exists a continuous solution $\tilde{\Phi}$ such that $\tilde{\Phi} = \Phi$ almost everywhere.
    \item If $f$ and $\phi$ are $C^k$, $f$ is strongly $r$-bunched for some $k \geq 2$ with $r < k - 1$, or $r = 1$, and $\Phi$ is a continuous solution of equation~\eqref{eq:cohomological}, then $\Phi$ is of class $C^r$.
\end{enumerate}
\end{proposition}

\begin{remark}
The original statement~\cite[Theorem~A]{wil13} is given for the case where $U = M$. As noted in~\cite[Section~12]{wil13}, the proof applies to partially hyperbolic diffeomorphisms on non-compact manifolds provided that the inequalities in the definitions of partial hyperbolicity and strong $r$-bunching hold uniformly. Therefore, the theorem extends directly to $f$-invariant open accessibility classes as above.
\end{remark}

\section{Lipschitz regularity of $E^s\oplus E^u$ and center dynamics}\label{sec acc and iso}
In this section, we show that the Lipschitz regularity of $E^s \oplus E^u$ imposes strong rigidity on the center dynamics. More precisely, the Lipschitz distribution $E^s \oplus E^u$ defines a contact form and the associated contact volume form. By comparing this contact volume with the $f$-invariant volume form, we prove that the cocycle $Df|_{E^c}$ is a coboundary on each open accessibility class.

Let $f \colon M \to M$ be a $C^2$ volume-preserving partially hyperbolic diffeomorphism on a 3-manifold, and denote by $m$ the volume form preserved by $f$. By considering a finite lift and iterate if necessary, we always assume that $f$ preserves the orientation of $E^c$. 

If $E^s \oplus E^u$ is Lipschitz, by Lemma~\ref{lemma alpha}, we can choose a Lipschitz nowhere-vanishing $1$-form $\alpha$ with $\ker \alpha = E^s \oplus E^u$. Since $f$ preserve $E^s\oplus E^u$ and the orientation of $E^c$, there is a positive Lipschitz function $\rho\colon M \to \RR^+$ such that 
\begin{equation}\label{f_etoile_alpha}
    f^\ast \alpha = \rho \alpha.
\end{equation} 
Since $\alpha$ is Lipschitz, it is differentiable $m$-almost everywhere. Let $\theta = \alpha \wedge d\alpha$;  then $\theta$ is a measurable $3$-form and there exists a measurable function $h \colon M \to \mathbb{R}$ such that $\theta_x = h(x) m_x$ for every $x \in M$. 
\begin{lemma}\label{lemma hrho}
We have $h(f(x)) = \rho^{2}(x) h(x)$, for $m$-almost every $x \in M$.
\end{lemma}
\begin{proof}
Since $m$ is $f$-invariant, for $m$-almost $x \in M$ we have
\[
(f^\ast \theta)_x = (f^\ast (h m))_x = h(f(x)) (f^\ast m)_x = h(f(x)) m_x.
\]
On the other hand, since
\[f^\ast \theta =f^\ast (\alpha \wedge d\alpha)=(f^\ast \alpha) \wedge d(f^\ast \alpha)=\rho \alpha \wedge d(\rho \alpha) =\rho^2 \theta, \]
we also have
\[
h(f(x)) m_x=(f^\ast \theta)_x = \rho^{2}(x) \theta_x = \rho^{2}(x) h(x) m_x,
\]
for $m$-almost every $x \in M$, which concludes the proof. 
\end{proof}

\begin{lemma}\label{lemma coboundary}
Let $f \colon M \to M$ be a $C^2$ volume-preserving partially hyperbolic diffeomorphism on a 3-manifold preserving the orientation of $E^c$. Suppose $E^s \oplus E^u$ is Lipschitz, then for any open accessibility class $U$ such that $f|_U$ is bijective and ergodic, the cocycle $\rho|_U$ is a continuous coboundary, i.e., there is a continuous positive function $\tilde h\colon U \to \RR^+$ such that $\tilde h(x)=|h(x)|$ for $m$-almost every $x \in U$, and  
\begin{equation}\label{eq rho}
    \rho(x) = \sqrt{\tilde h(f(x))}/\sqrt{\tilde h(x)}, \qquad \forall\, x \in U.
\end{equation}

\end{lemma}

\begin{proof}
Let $A^+:=\{x \in M:h(x) >0\}$,  $A^-:=\{x \in M:h(x)<0\}$ and $A=A^+\cup A^-$. By Lemma~\ref{lemma hrho}, the sets  $A^+,A^-$ and $A$ are all essentially $f$-invariant. Since $f$ is ergodic, $A$ has either full or zero measure. If $m(A) = 0$, then $E^s \oplus E^u$ is integrable by Proposition~\ref{prop frobenius}, which is a contradiction. 

Therefore, either $A^+$ or $A^-$ have full measure. Without loss of generality, we assume that $A^+$ has full measure. Then by Lemma~\ref{lemma hrho}, $\sqrt{h(x)}$ is a measurable solution to the cohomological equation~\eqref{eq rho}. By Proposition~\ref{prop wil13}, there is a continuous function $\tilde h\colon U \to \RR^+$  solving~\eqref{eq rho}, with $\tilde h=h$ for $m$-almost every $x \in U$. 
\end{proof}

We record two direct corollaries that will be used later.
\begin{corollary}\label{coro center iso}
Let $f$ be as in Lemma~\ref{lemma coboundary}. Then $x \mapsto \lambda_x^c := \|Df|_{E^c(x)}\|$ is a coboundary on $U$.
\end{corollary}
\begin{proof}
Let $X$ be the vector field such that  $X_x\in E^c(x)$ and $\alpha_x(X_x) = 1$,  for every $x \in M$. We choose a Riemannian metric $\|\cdot\|$ such that $\|X_x\| = 1$ for all $x \in M$. Then for every $x \in U$,
\[
\lambda_x^c
= \alpha_{f(x)}\bigl(Df|_{E^c(x)}(X_x)\bigr)
= (f^\ast \alpha)_x(X_x)
= \rho(x)= \sqrt{\tilde h(f(x))}/\sqrt{\tilde h(x)}
\]
is a coboundary. Note that the property that $\lambda_x^c|_U$ is a coboundary does not rely on the choice of a metric;  this finishes the proof. 
\end{proof}\begin{remark}
We will show later that if $f$ is as in Lemma~\ref{lemma coboundary}, then $U=M$ (see Proposition~\ref{prop dic}). In particular, by Corollary~\ref{coro center iso}, the center Lyapunov exponent of $f$ vanishes for every ergodic measure.
\end{remark}
\begin{corollary}\label{coro ec smooth}
Let $f$ be as in Lemma~\ref{lemma coboundary}. If $f$ is $C^k$ and strongly $r$-bunched, and $E^s \oplus E^u$ is $C^k$, for some $k \geq 2$ with $r < k - 1$ or $r = 1$, then $f|_U$ preserves a $C^{r}$ contact form $\beta$ with $\ker \beta = E^s \oplus E^u$ and whose Reeb vector field is tangent to $E^c$. In particular, $E^c$ is $C^{r-1}$.
\end{corollary}

\begin{proof}
Since $E^s \oplus E^u$ is $C^k$, the $1$-form $\alpha$ and the function $\rho$ are $C^k$. Since $h$ is a continuous solution of the cohomological equation~\eqref{eq rho}, Proposition~\ref{prop wil13} implies that $h$ is $C^r$. Let $\beta := \alpha / \sqrt{\tilde h}$. Then $\beta$ is a $C^r$ contact form and by~\eqref{eq rho}, 
\[
(f^\ast \beta)_x
= \frac{(f^\ast \alpha)_x}{\sqrt{\tilde h(f(x))}}
= \frac{\rho(x)\alpha_x}{\sqrt{\tilde h(f(x))}}
= \beta_x.
\]
Thus, $\beta$ is $f$-invariant, and so is its Reeb vector field $R$.

We now show that $R_x \subset E^c(x)$, for all $x \in M$. Let us argue by contradiction. We may assume that $R$ has a non-trivial component in $E^u$ at an $m|_U$-typical point $x$. Choose an open neighborhood $V \subset U$ of $x$ such that $\|R\|$ is bounded on $V$. By ergodicity, there exists a sequence $n_k \to \infty$ such that $f^{n_k}(x) \in V$. Since $R$ is $f$-invariant, we must have $\|R_{f^{n_k}(x)}\| \to \infty$, which is a contradiction.

In particular, $E^c$ is a $C^{r-1}$ distribution, because it is spanned by the $C^{r-1}$ Reeb vector field $R$.
\end{proof}

\section{Foulon--Hasselblatt cocycle and the regularity of the $E^s\oplus E^u$ distribution}\label{sec_FH_cocycle}

Let us recall that for a conservative 3-dimensional Anosov flow, Foulon and Hasselblatt~\cite{fh03} introduced a cocycle --- which they call the~\emph{longitudinal KAM cocycle} --- which characterizes the regularity of the distribution $E^s \oplus E^u$. This cocycle is essentially defined by considering the second-order derivative (along the stable/unstable directions) of the hitting times of the flow with respect to a family of uniformly smooth transverse sections. While this cocycle is~\emph{a priori} only defined for continuous dynamical systems, Gogolev--Leguil--Rodriguez Hertz~\cite{glh25} observed that it can be generalized using the templates constructed by Tsujii--Zhang~\cite{tz23} and Eskin--Potrie--Zhang~\cite{epz}. 
Our goal in what follows is to define this cocycle for 3-dimensional volume-preserving partially hyperbolic diffeomorphisms and explore its relationship with the regularity of the distribution $E^s \oplus E^u$ in this context.

\subsection{Adapted charts and templates}
Let $r \in \mathbb{N}_{\geq 5}\cup \{\infty\}$. 
In the following, we consider a volume-preserving $C^r$ partially hyperbolic diffeomorphism $f\colon M \to M$ of a $3$-manifold $M$, with splitting $TM=E^s \oplus E^c \oplus E^u$. Recall that we denote by $\mathcal{W}^s$ and $\mathcal{W}^u$ the stable and unstable foliations respectively. For any $x \in M$ and $*=s,c,u$,  we denote 
\[\lambda^\ast_x(n):=\|D f^n(x)|_{E^\ast}\|,\quad \forall\, n \in \ZZ.\]
We also abbreviate $\lambda^\ast_x:=\lambda^\ast_x(1)$. 

Let us recall the non-stationary normal forms which linearize the dynamics along $1$-dimensional stable and unstable foliations:
\begin{proposition}[see e.g.~Katok-Lewis~\cite{kl91}]
	\label{norm foms}
	For $*=s,u$, there exists a continuous family of $C^r$ charts $\{\Phi_x^*\colon T_x \mathcal{W}^*(x)\to \mathcal{W}^*(x)\}_{x\in M}$ such that for any $x\in M$, we have 
	\begin{enumerate}
		\item $\Phi_x^*(0)=x$, and $D\Phi_x^*(0)=\mathrm{Id}$;
		\item $f(\Phi_x^*(\xi))=\Phi_{f(x)}^*(\lambda_x^* \xi)$, for any $\xi \in \mathbb{R}$. 
	\end{enumerate}
\end{proposition}

Now we recall the construction of adapted charts, due to Tsujii~\cite{tsu18} for $3$-dimensional volume-preserving Anosov flows, to Tsujii--Zhang~\cite{tz23} for general $3$-dimensional Anosov flows, and later extended to the partially hyperbolic setting by Eskin--Potrie--Zhang~\cite{epz}.

 
\begin{proposition}[Adapted charts]\label{propo o good}
Let $f$ be a $C^r$ volume-preserving $3$-dimensional partially hyperbolic diffeomorphism as above, $r \geq 5$. Then, there exists 
a continuous family of  uniformly $C^{r-1}$ charts $\{\imath_x\colon (-1,1)^3\to M\}_{x \in M}$ such that for any $x \in M$, we have:
	\begin{enumerate}
		\item\label{normal un} $\imath_x(\xi,0,0)=\Phi_x^s(\xi)$, for any $\xi \in (-1,1)$; 
		\item\label{normal deux} $\imath_x(0,0,\eta)=\Phi_x^u(\eta)$, for any $\eta \in (-1,1)$; 
		\item\label{flw dir} $\partial_2\imath_x(\xi,0,\eta)\in E^c(x)$, for any $(\xi,\eta)\in (-1,1)^2$;
		\item\label{pt cocyc nf} let $F_x:=(\imath_{f(x)})^{-1}\circ f\circ \imath_x=(F_{x,1},F_{x,2},F_{x,3})$; then, 
        there exist quadratic polynomials $P_x^s (z)=
        \alpha_x^{\mathrm{FH}} z+\beta_x^s z^2$, $P_x^u (z)=
        \alpha_x^{\mathrm{FH}} z+\beta_x^u z^2$, 
        such that for $\xi\in (-1,1)$ and $\eta \in (-(\lambda_x^u)^{-1},(\lambda_x^u)^{-1})$, we have  
		\begin{align*} 
			\begin{bmatrix}
				\partial_1 F_{x,1} & \partial_2 F_{x,1}\\
				\partial_1 F_{x,2} & \partial_2 F_{x,2} 
			\end{bmatrix}(0,0,\eta)&=\begin{bmatrix}
				\lambda_x^s  & 0\\
				P_x^s(\eta) & \lambda_x^c
			\end{bmatrix},\\
			\begin{bmatrix}
				\partial_2 F_{x,2} & \partial_3 F_{x,2}\\
				\partial_2 F_{x,3} & \partial_3 F_{x,3} 
			\end{bmatrix}(\xi,0,0)&=\begin{bmatrix}
				\lambda_x^c & P_x^u(\xi)\\
				0 & \lambda_x^u 
			\end{bmatrix}. 
        \end{align*}
        In particular,  
        \begin{equation}\label{def_FHcoc}
            \alpha_x^{\mathrm{FH}}:=\partial_{13} F_{x,2}(0,0,0),\quad \forall\, x \in M.
        \end{equation}
We emphasize that the notation $\alpha_x^{\mathrm{FH}}$ here is unrelated to the contact form $\alpha$ introduced in the previous section.
	\end{enumerate}
\end{proposition}
\begin{remark}\label{remarque_dis}
    While not required for the current paper, such adapted charts can in fact be defined for any $C^\infty$ partially hyperbolic diffeomorphism on a $3$-manifold, even in the dissipative case. More generally, this construction extends to $C^r$ partially hyperbolic diffeomorphisms satisfying an appropriate pinching condition (see, e.g.,~\cite{glh25}). In these cases, however, the polynomials appearing in the off-diagonal entries of the differentials may have higher degrees.
\end{remark}
\begin{proof}[Proof of Proposition~\ref{propo o good}]
Detailed proofs are provided in~\cite{tz23,epz}; for an explanation of why the charts can be chosen to be $C^{r-1}$ (and not only $C^{r-2}$ as in~\cite{tz23}), see also~\cite{glh25}. Below, we only sketch the proof, focusing on why the polynomials in the preceding formulas are quadratic.

By Proposition~\ref{norm foms}, there exist non-stationary linearizing charts $\{\Phi_x^s\}_{x \in M}$ and $\{\Phi_x^u\}_{x \in M}$ along the stable and unstable manifolds, respectively. The construction of these charts ensures that their regularity matches that of the diffeomorphism $f$, i.e., they are $C^r$. It is standard to extend these non-stationary linearizations to full $3$-dimensional charts $\{\jmath_x\colon (-1,1)^3 \to M\}_{x \in M}$ such that the first three properties of Proposition~\ref{propo o good} are satisfied.

The proof then involves a series of chart adjustments. By solving appropriate cohomological equations (see, e.g.,~\cite{epz}), we obtain a family $\{\hat{h}_x\}_{x \in M}$ of uniformly $C^{r-1}$ diffeomorphisms such that the adjusted charts $\hat{\jmath}_x := \jmath_x \circ \hat{h}_x$ place the dynamics in the form $\hat{F}_x:=(\hat{\jmath}_{f(x)})^{-1}\circ f\circ \hat{\jmath}_x=(\hat{F}_{x,1},\hat{F}_{x,2},\hat{F}_{x,3})$, where for $\xi\in (-1,1)$ and $\eta \in (-(\lambda_x^u)^{-1},(\lambda_x^u)^{-1})$,  
    \begin{align*} 
			\begin{bmatrix}
				\partial_1 \hat{F}_{x,1} & \partial_2 \hat{F}_{x,1}\\
				\partial_1 \hat{F}_{x,2} & \partial_2 \hat{F}_{x,2} 
			\end{bmatrix}(0,0,\eta)&=\begin{bmatrix}
				\lambda_x^s  & 0\\
				r_x^s(\eta) & \lambda_x^c
			\end{bmatrix},\\
			\begin{bmatrix}
				\partial_2 \hat{F}_{x,2} & \partial_3 \hat{F}_{x,2}\\
				\partial_2 \hat{F}_{x,3} & \partial_3 \hat{F}_{x,3} 
			\end{bmatrix}(\xi,0,0)&=\begin{bmatrix}
				\lambda_x^c & r_x^u(\xi)\\
				0 & \lambda_x^u 
			\end{bmatrix}, 
        \end{align*}
for two $C^{r-2}$ functions $r_x^s,r_x^u$. 

To obtain the desired polynomials as stated in Proposition~\ref{propo o good}, two additional 
cohomological equations must be solved. We focus on the case of $r_x^u$; the other case follows similarly. We decompose $r_x^u(\xi)$ as $r_x^u(\xi) = P_x^u(\xi) + \hat{r}_x^u(\xi)$, where $P_x^u(\xi) = \alpha_x^{\mathrm{FH}} \xi + \beta_x^u \xi^2$ and $\hat{r}_x^u(\xi) = O(\xi^3)$. 
As in~\cite[Claim 3.5]{epz}, the required chart adjustment involves solving the following cohomological equation for $u$:
\[
u_{f(x)}(\lambda_x^s \xi) = \frac{1}{\lambda_x^u} \left( \hat{r}_x^u(\xi) + \lambda_x^c u_x(\xi) \right).
\]
This can be rewritten as
\[
u_x(\xi) - \frac{\lambda_x^u}{\lambda_x^c} u_{f(x)}(\lambda_x^s \xi) = -\frac{1}{\lambda_x^c} \hat{r}_x^u(\xi).
\]
By telescopic summation, a formal solution is given by
\[
u_x(\xi) = -\sum_{k=0}^{+\infty} \frac{1}{\lambda^c_{f^k(x)}} \frac{\lambda_x^u(k-1)}{\lambda_x^c(k-1)} \hat{r}_{f^k(x)}^u(\lambda_x^s(k)\xi), \quad \forall\, \xi \in (-1,1).
\]
In fact, since $\hat{r}_x^u(\xi) = O(\xi^3)$, the general term of the series decays as
\[
\frac{\lambda_x^u(k)(\lambda_x^s(k))^3}{\lambda_x^c(k)} = \lambda_x^u(k)\lambda_x^c(k)\lambda_x^s(k) \left( \frac{\lambda_x^s(k)}{\lambda_x^c(k)} \right)^2.
\]
Under our volume-preserving assumption, the cocycle $(x,n) \mapsto \lambda_x^u(n)\lambda_x^c(n)\lambda_x^s(n)$ is bounded. Meanwhile, by partial hyperbolicity, the term $\frac{\lambda_x^s(k)}{\lambda_x^c(k)}$ decays exponentially with $k$. Thus, the series converges and defines a $C^{r-2}$ solution to the cohomological equation. Following the argument in~\cite{glh25}, we can further construct a $C^{r-1}$ chart adjustment that transforms the differential of $f$ into the desired form, as given in point~\eqref{pt cocyc nf} of Proposition~\ref{propo o good}. 
\end{proof}
We observe that in adapted charts, along the coordinate axes $\{(\xi,0,0)\}_{\xi \in (-1,1)}$ and $\{(0,0,\eta)\}_{\eta \in (-1,1)}$, the distribution $E^s\oplus E^u$ is controlled by two angles $\mathcal{T}_x^u$, $\mathcal{T}_x^s$, called respectively the~\emph{unstable and stable templates}; more precisely,
\begin{itemize}
    \item for any $\xi \in (-1,1)$, $\mathcal{T}_x^u(\xi)$ is the ``slope'' of the distribution $(E^s \oplus E^u)(\Phi_x^s(\xi))$ with respect to the horizontal plane $\{t=0\}$, i.e., 
    $$
    (E^s \oplus E^u)(\Phi_x^s(\xi))=D\imath_x(\xi,0,0)\left(\mathbb{R}\frac{\partial}{\partial \xi} \oplus \mathbb{R}(0,\mathcal{T}_x^u(\xi),1)^\top\right). 
    $$
    \item for any $\eta \in (-1,1)$, $\mathcal{T}_x^s(\eta)$ is the ``slope'' of the distribution $(E^s \oplus E^u)(\Phi_x^u(\eta))$ with respect to the horizontal plane $\{t=0\}$, i.e., 
    $$
    (E^s \oplus E^u)(\Phi_x^u(\eta))=D\imath_x(0,0,\eta)\left( \mathbb{R}(1,\mathcal{T}_x^s(\eta),0)^\top\oplus \mathbb{R}\frac{\partial}{\partial \eta} \right). 
    $$
\end{itemize}
In normal coordinates, the invariance of $E^s \oplus E^u$ under the differential $Df$ yields:
\begin{lemma} 
	For any $x \in M$, $\xi \in (-1,1)$ and $\eta \in (-(\lambda_x^u)^{-1},(\lambda_x^u)^{-1})$,  
	we have 
	\begin{align} 
		P_x^s(\eta) &=\lambda^s_x\mathcal{T}^s_{f(x)}(\lambda_x^u\eta)-\lambda_x^c\mathcal{T}^s_{x}(\eta),\label{stable_templ_eq} \\
        P_x^u(\xi) &=\lambda^u_x\mathcal{T}^u_{f(x)}(\lambda_x^s\xi)-\lambda_x^c\mathcal{T}^u_{x}(\xi). \label{unstable_templ_eq}
	\end{align} 
\end{lemma}

\begin{proof}
We show it for $\mathcal{T}_x^s$, the other case is analogous. For $\eta \in (-(\lambda_x^u)^{-1},(\lambda_x^u)^{-1})$, the image of $(1,\mathcal{T}^s_{x}(\eta))^\top$ by the differential $Df(x)$ in normal coordinates, namely
	$$
	\begin{bmatrix}
		\lambda_x^s & 0\\
		P_x^s(\eta) & \lambda_x^c
	\end{bmatrix}\begin{bmatrix}
		1 \\
		\mathcal{T}^s_{x}(\eta)
	\end{bmatrix}= \begin{bmatrix}
		\lambda_x^s\\
		P_x^s(\eta) +\lambda_x^c\mathcal{T}^s_{x}(\eta)
	\end{bmatrix},
	$$
	is still in $D\imath_{f(x)}(0,0,\lambda_x^u\eta)^{-1}(E^s \oplus E^u)(\Phi_{f(x)}^u(\lambda_x^u\eta))$, and thus is proportional to $(1,\mathcal{T}^s_{f(x)}(\lambda_x^u\eta))^\top$, which yields
	\begin{equation*}
		P_x^s(\eta)=\lambda^s_x\mathcal{T}^s_{f(x)}(\lambda_x^u\eta)-\lambda_x^c\mathcal{T}^s_{x}(\eta).\qedhere
	\end{equation*}\end{proof}
\subsection{The Foulon--Hasselblatt cocycle}
Let us denote by $\lambda\colon M \times \mathbb{Z} \ni (x,n)\mapsto \lambda_x (n)$ the multiplicative cocycle associated to the map $x \mapsto \lambda_x=\frac{\lambda_x^s \lambda_x^u}{\lambda_{f(x)}^c}$, specifically, $\lambda_x(n):=\frac{\lambda_x^s(n) \lambda_x^u(n)}{\lambda_{f(x)}^c(n)}$. 
\begin{definition}
We define the (twisted)~\emph{Foulon--Hasselblatt cocycle} $\alpha^{\mathrm{FH}}$ to be the map 
\begin{equation*} 
    M \times \mathbb{Z} \ni (x,n)\mapsto \alpha_x^{\mathrm{FH}} (n):=
    \left\{
    \begin{array}{ll}
	\sum_{\ell=0}^{n-1} \lambda_x(\ell)\alpha_{f^\ell(x)}^{\mathrm{FH}},& n \geq 0,\\
    -\sum_{\ell=1}^{-n} \lambda_x(-\ell)\alpha_{f^{-\ell}(x)}^{\mathrm{FH}},& n < 0,
    \end{array}
    \right.
\end{equation*}
with $\alpha_x^{\mathrm{FH}}$ as in~\eqref{def_FHcoc}. 
Note that the twist is given by the cocycle $\lambda$ defined above. 
\end{definition}

\begin{remark}
    For $C^\infty$ dissipative partially hyperbolic diffeomorphisms on $3$-manifolds, Remark~\ref{remarque_dis} ensures that adapted charts can still be defined. Consequently, the Foulon--Hasselblatt cocycle can be defined analogously to expression~\eqref{def_FHcoc}.
\end{remark}
\begin{remark}
The Foulon--Hasselblatt cocycle $\alpha^{\mathrm{FH}}$ is a~\emph{twisted coboundary} if 
\[\alpha_x^{\mathrm{FH}}=\lambda_x\beta(f(x))-\beta(x),\quad \forall\, x \in M,\]
for some continuous function $\beta\colon M \to \RR$. 
\end{remark}
\begin{lemma}
    The twisted cohomology class of the Foulon--Hasselblatt cocycle is independent of the choice of adapted charts. In other words, if $\{\tilde \imath_x\}_{x \in M}$ is another family of adapted charts, and $\tilde{\alpha}^{\mathrm{FH}}$ denotes the associated Foulon--Hasselblatt cocycle, then the difference $\tilde{\alpha}^{\mathrm{FH}}-\alpha^{\mathrm{FH}}$ is a twisted coboundary, i.e., 
$$
\tilde{\alpha}_x^{\mathrm{FH}}-\alpha_x^{\mathrm{FH}}=\lambda_x\beta(f(x))-\beta(x),\quad \forall\, x \in M,
$$
for some continuous function $\beta \colon M \to \mathbb{R}$. 
    \end{lemma}
\begin{proof}
    For any $x \in M$, we let $\Phi_x:=\tilde \imath_x^{-1}\circ \imath_x\colon (\xi,t,\eta)\mapsto (\tilde \xi_x,\tilde t_x,\tilde \eta_x)$ be the change of coordinates. By our requirements on the adapted charts, we have $\Phi_x|_{\{(0,0)\}\times (-1,1)}=\mathrm{Id}$ and $\Phi_x|_{(-1,1)\times \{(0,0)\}}=\mathrm{Id}$. Let $\tilde F_x:=(\tilde\imath_{f(x)})^{-1}\circ f\circ \tilde\imath_x=(\tilde F_{x,1},\tilde F_{x,2},\tilde F_{x,3})$, so that, for $\eta \in (-(\lambda_x^u)^{-1},(\lambda_x^u)^{-1})$,   
    \begin{align*} 
			\begin{bmatrix}
				\partial_1 \tilde F_{x,1} & \partial_2 \tilde F_{x,1}\\
				\partial_1 \tilde F_{x,2} & \partial_2 \tilde F_{x,2} 
			\end{bmatrix}(0,0,\eta)&=\begin{bmatrix}
				\lambda_x^s  & 0\\
				\tilde \alpha_x^{\mathrm{FH}} \eta+\tilde{\beta}_x^s \eta^2 & \lambda_x^c
			\end{bmatrix}.
        \end{align*}
    For $|\xi|,|\eta|\ll 1$, we have 
    $$
    \Phi_x(\xi,0,\eta)=(\xi+H.O.T.,\tau_x \xi\eta+H.O.T., \eta+H.O.T.),
    $$
    where $\tau_x:=\partial_{13} \tilde t_x(0,0,0)$. By definition,   $\tilde{F}_x\circ \Phi_x=\Phi_{f(x)} \circ F_x$. On the one hand, we have 
    \begin{align*}
    \tilde{F}_{x,2}\circ \Phi_x(\xi,0,\eta)&=\tilde{F}_{x,2}(\xi,\tau_x \xi\eta, \eta)+H.O.T.\\
    &=\tilde{F}_{x,2}(\xi,0,\eta)+\partial_2 F_{x,2}(\xi,0,\eta)\tau_x \xi \eta+H.O.T.\\
    &=(\tilde{\alpha}_x^{\mathrm{FH}}+\lambda_x^c \tau_x) \xi \eta+ H.O.T.
    \end{align*}
    On the other hand, we have 
    \begin{align*}
    \Phi_{f(x)}\circ \tilde{F}_{x}(\xi,0,\eta)&=\Phi_{f(x)}(\lambda_x^s \xi+H.O.T.,\alpha_x^{\mathrm{FH}} \xi\eta+H.O.T., \lambda_x^u \eta+H.O.T.)\\
    &=(*,(\alpha_x^{\mathrm{FH}}+\tau_{f(x)}\lambda_x^s \lambda_x^u) \xi\eta +H.O.T.,*). 
    \end{align*}
    By comparing the two expressions, we deduce that
   \begin{equation}\label{change_FH_cocy}
   \tilde{\alpha}_x^{\mathrm{FH}}-\alpha_x^{\mathrm{FH}}=\lambda_x^s \lambda_x^u\tau_{f(x)}-\lambda_x^c\tau_x=\lambda_x \beta(f(x))-\beta(x), 
   \end{equation}
    letting $\beta(y):=\lambda_y^c \tau_y=\lambda_y^c \partial_{13} \tilde t_y(0,0,0)$, for $y \in M$. 
\end{proof}
Moreover, if the Foulon--Hasselblatt cocycle is a twisted coboundary, then by
changing the adapted charts one can normalize it to vanish identically.

\begin{lemma}\label{vanish_FH}
Assume that the Foulon--Hasselblatt cocycle is a twisted coboundary. Then, we can find a family $\{\tilde \imath_x\}_{x \in M}$ of adapted charts in which the Foulon--Hasselblatt cocycle vanishes identically. 
    \end{lemma}
\begin{proof}
Assume that for some continuous function $\beta \colon M \to \mathbb{R}$, we have 
$$
\alpha_x^{\mathrm{FH}}=\lambda_x\beta(f(x))-\beta(x),\quad \forall\, x \in M. 
$$
    We look for a family $\{\tilde \imath_x\}_{x \in M}$ such that the associated family of (quadratic) polynomials $\tilde P_x^*$ have a vanishing linear coefficient, i.e., $\tilde \alpha_x^{\mathrm{FH}}=0$. As above, we let $\Phi_x:=\tilde \imath_x^{-1}\circ \imath_x\colon (\xi,t,\eta)\mapsto (\tilde \xi_x,\tilde t_x,\tilde \eta_x)$ be the change of coordinates; by~\eqref{change_FH_cocy}, we then have 
   $$
   \tilde \alpha_x^{\mathrm{FH}}=\alpha_x^{\mathrm{FH}}+\lambda_{x} \lambda_{f(x)}^c \partial_{13}\tilde t_{f(x)}(0,0,0)-\lambda_x^c\partial_{13}\tilde t_x(0,0,0).
   $$
   To ensure that $\tilde \alpha_x^{\mathrm{FH}}=0$, it suffices to solve $\lambda_x^c\partial_{13}\tilde t_x(0,0,0)=\beta(x)$; we can then take  
   $$
   \tilde \imath_x(\xi,t,\eta)=\imath_x\left(\xi,t-\frac{\beta(x)}{\lambda_x^c} \xi\eta,\eta\right),\quad \forall\, x \in M,\, \forall\, \xi,t,\eta\in (-1,1). \qedhere
   $$ 
\end{proof}
The non-triviality of the Foulon--Hasselblatt cocycle $\alpha^{\mathrm{FH}}$ is an obstruction to the regularity of the distribution $E^s \oplus E^u$. 
\begin{lemma}\label{lemme_cond_nec}
    If $E^s \oplus E^u$ is $C^1$ along the leaves of the unstable foliation $\mathcal{W}^u$, then $\alpha^{\mathrm{FH}}$ is a twisted coboundary. 
The same conclusion holds if $E^s \oplus E^u$ is $C^1$ along the leaves of the stable foliation $\mathcal{W}^s$.  
\end{lemma}
\begin{proof}
    We first observe that for any point $x \in M$, the distribution $E^s \oplus E^u$ is $C^1$ along $\mathcal{W}^u(x)$ if and only if the stable template $\mathcal{T}_x^s$ is $C^1$. In particular, if the latter is $C^1$ along $\mathcal{W}^u(x)$, then by differentiating~\eqref{stable_templ_eq} at $\eta=0$, we have 
    \begin{align*}
    \alpha_x^{\mathrm{FH}}&=\lambda_x^s \lambda_x^u(\mathcal{T}_{f(x)}^s)'(0)-\lambda_x^c(\mathcal{T}_x^s)'(0)\\
    &=\lambda_x\beta(f(x))-\beta(x),
    \end{align*}
with $\beta \colon x \mapsto \lambda_x^c(\mathcal{T}_x^s)'(0)$. 
\end{proof}
\subsection{Bootstrap of regularity}  
Although the Foulon--Hasselblatt cocycle can be defined for any $C^\infty$ partially hyperbolic diffeomorphism on 3-manifold, we focus on its application in a specific case. 

To be precise, we assume that $f\colon M \to M$ is a  $C^\infty$~\emph{dynamically coherent} and~\emph{volume-preserving} partially hyperbolic diffeomorphism on a 3-manifold $M$. By Proposition~\ref{prop ham17hp25}, after taking a finite iterate if necessary, we can assume that every open accessibility class $U$ (if it exists) is $f$-invariant. 

\begin{lemma}\label{lemma lip boot}
    If $E^s \oplus E^u$ is Lipschitz, then $E^s \oplus E^u$ is $C^1$. 
\end{lemma}
\begin{proof}
Note that $f$ is dynamically coherent and center bunched, hence the distribution $E^s \oplus E^u$ is uniformly $C^1$ along $\W^c$. It therefore suffices to show that $E^s \oplus E^u$ is uniformly $C^1$ along the local manifolds $\W^s_{\mathrm{loc}}(x)$ and $\W^u_{\mathrm{loc}}(x)$ for all $x \in M$. By construction, this is equivalent to proving that the templates $\mathcal{T}^s_x$ and $\mathcal{T}^u_x$ are uniformly $C^1$. We prove this for $\mathcal{T}^s_x$; the argument for $\mathcal{T}^u_x$ is analogous.

We first prove that $\mathcal{T}^s_x$ is uniformly $C^1$ on every open accessibility class $U$. Since $E^s \oplus E^u$ is Lipschitz, the stable template $\mathcal{T}_x^s$ is Lipschitz and is differentiable at $0$ for $m$-almost $x \in M$. Then by differentiating~\eqref{stable_templ_eq} at $\eta=0$, we have 
\[\alpha_x^{\mathrm{FH}}=\lambda_x^s \lambda_x^u(\mathcal{T}_{f(x)}^s)'(0)-\lambda_x^c(\mathcal{T}_x^s)'(0)=\lambda_x\beta(f(x))-\beta(x)\]
for $m$-almost $x \in M$ with $\beta \colon x \mapsto \lambda_x^c(\mathcal{T}_x^s)'(0)$. Since $f$ is volume-preserving, the multiplicative cocycle $x \mapsto \lambda_x^s\lambda_x^u\lambda_x^c$ is a coboundary. By Corollary~\ref{coro center iso}, the multiplicative cocycle $x \mapsto \lambda^c_x|_U$ is also a coboundary.  Therefore, the product 
    \[    x \mapsto \lambda_x=\frac{\lambda_x^s\lambda_x^u}{\lambda_{f(x)}^c}=\left(\lambda_x^s\lambda_x^u\lambda_x^c\right) \frac{\lambda_x^c}{\lambda_{f(x)}^c}(\lambda_x^c)^{-2}\]
is a coboundary on $U$. In particular, we have $\lambda_x = \gamma(f(x))/\gamma(x)$ for some continuous function $\gamma\colon U \to \RR$, and then 
\[\alpha'_x:=\gamma(x)\alpha_x^{\mathrm{FH}}=\gamma(x)(\lambda_x\beta(f(x))-\beta(x)) =\beta(f(x))\gamma(f(x))-\beta(x)\gamma(x)\]
for $m$-almost $x \in U$. This implies that $\alpha'_x$ is a measurable coboundary. By Proposition~\ref{prop wil13}, there is a continuous function $\tilde \beta\colon U \to \RR$ such that $\tilde \beta(x) =\beta(x)$ for $m$-almost every $x \in U$ and 
\[\alpha'_x=\tilde \beta(f(x))\gamma(f(x))-\tilde \beta(x)\gamma(x), \quad \forall\, x\in U.\]
In particular, $(\mathcal{T}_x^s)'(0)=\tilde \beta(x)/\lambda^c_x$ can be extended to a continuous function on $U$. This implies that $\mathcal{T}^s_x$ is uniformly $C^1$ on $U$. 

We then prove $\mathcal{T}^s_x$ is uniformly $C^1$ when $x \in S$ for some 2-torus tangent to $E^s\oplus E^u$. By Lemma~\ref{lemma ks06}, the 2-torus $S$ tangent to $E^s\oplus E^u$ is a $C^\infty$ submanifold. Therefore, the template $\mathcal{T}^s_x$ is uniformly $C^1$ (actually uniformly $C^\infty$) along $S$. Note that the adapted chart $\{\imath_x\}_{x \in M}$ is uniformly $C^\infty$ and $E^s \oplus E^u$ is uniformly $C^1$ along $\W^c$, the templates $\mathcal{T}^s_x$ change by composition with a family of $C^1$ functions when $x$ moves along $\W^c$. In particular, if $(\mathcal{T}^s_x)'(0)$ is well-defined, it is automatically continuous along $\W^c$. This implies $\mathcal{T}^s_x$ is uniformly $C^1$ for every $x \in S$. 
\end{proof}

We then present the main result of this section. In the following, $U$ will denote an open accessibility class of $f$. By the previous discussion, we may assume without loss of generality that $U$ is $f$-invariant. 
\begin{proposition}\label{lemma bootsrap}
If $E^s \oplus E^u$ is Lipschitz and $f|_{\bar{U}}$ is $\infty$-bunched, then $E^s \oplus E^u$ is $C^\infty$ on $U$.
\end{proposition}

\begin{proof}
Since $f|_{\bar{U}}$ is $\infty$-bunched, the leaves of $\W^c$ are $C^\infty$ and $E^s \oplus E^u$ is $C^\infty$ along $\W^c$. By Journé's Lemma~\cite{jou88}, in order to show that $E^s \oplus E^u$ is $C^\infty$ on $U$, it suffices to show it is (uniformly) $C^\infty$ along the leaves of $\mathcal{W}^s|_U$ and $\mathcal{W}^u|_U$ respectively. As we observed above, $E^s\oplus E^u$ is $C^\infty$ along the leaves of $\mathcal{W}^u|_U$ if and only if the family of stable templates $\{\mathcal{T}_x^s\}_{x \in U}$ are uniformly $C^\infty$. Similarly, $E^s\oplus E^u$ is $C^\infty$ along the leaves of $\mathcal{W}^s|_U$ if and only if the family of unstable templates $\{\mathcal{T}_x^u\}_{x \in U}$ are uniformly $C^\infty$. We only deal with the case of stable templates; the other one is similar. 

By Lemma~\ref{lemma lip boot}, the distribution $E^s\oplus E^u$ is $C^1$. This implies that $x \mapsto \alpha_x^{\mathrm{FH}}$ is a twisted coboundary by Lemma~\ref{lemme_cond_nec}. Then by Lemma~\ref{vanish_FH}, up to replacing the charts $\{\imath_x\}_{x \in M}$ by new charts $\{\tilde{\imath}_x\}_{x \in M}$, we can assume that the Foulon--Hasselblatt cocycle vanishes identically. In other words,  for all $x \in M$, the degree-one coefficient $\alpha_x^{\mathrm{FH}}=(P_x^s)'(0)=0$ of the polynomial $P_x^s$ vanishes; consequently, $P_x^s(\eta)=\beta_x^s \eta^2$.  
    By~\eqref{stable_templ_eq}, for $x \in M$ and $\eta\in (-1,1)$, we have 
    $$
    \mathcal{T}_x^s(\eta)-\frac{\lambda_x^s(-1)}{\lambda_x^c(-1)}\mathcal{T}_{f^{-1}(x)}^s(\lambda_x^u(-1)\eta)=\lambda_{x}^s(-1)P_{f^{-1}(x)}^s(\lambda_x^u(-1)\eta).
    $$
    By iterating the previous relation, for every $x \in M$ and any integer $n \geq 0$, we thus obtain 
    \begin{equation}\label{iter_relat}
    \mathcal{T}_x^s(\eta)-\frac{\lambda_x^s(-n)}{\lambda_x^c(-n)}\mathcal{T}_{f^{-n}(x)}^s(\lambda_x^u(-n)\eta)=\sum_{\ell=1}^{n} \frac{\lambda_{x}^s(-\ell)}{\lambda_x^c(-(\ell-1))}P_{f^{-\ell}(x)}^s(\lambda_x^u(-\ell)\eta).
    \end{equation}
    The templates $\{\mathcal{T}_x^s\}_{x \in M}$ are uniformly bounded, hence by~\eqref{stable_templ_eq}, so are the coefficients of the polynomials $\{P_x^s\}_{x \in M}$. Since $f$ is volume-preserving, the map $x \mapsto \lambda^s_x\lambda^c_x\lambda^u_x$ is a coboundary. We deduce that for every $x \in M$, there exist constants $C,C'>0$ such that
    \begin{equation}\label{eq 1term}
           \left|\frac{\lambda_{x}^s(-\ell)}{\lambda_x^c(-(\ell-1))}P_{f^{-\ell}(x)}^s(\lambda_x^u(-\ell)\eta)\right|\leq C \frac{\lambda_{x}^s(-\ell)\lambda_x^u(-\ell)}{\lambda_x^c(-(\ell-1))}\lambda_x^u(-\ell)\leq C'\frac{\lambda^u_x(-\ell)}{\lambda_x^c(-(\ell-1))\lambda_x^c(-\ell)}.
    \end{equation}
By Corollary~\ref{coro center iso}, the center cocycle $x \mapsto \lambda_x^c$ restricted to $U$ is a coboundary, i.e., there exists a continuous function $\Phi\colon U \to \mathbb{R}\setminus \{0\}$ such that 
$$
\lambda_x^c(n)=\frac{\Phi(f^n(x))}{\Phi(x)},\quad \forall\, x \in U.
$$
Since $f|_U$ is accessible and ergodic, for $m$-a.e. $x \in U$, we can choose a subsequence $(n_j)_j$ of integers $n_j \to +\infty$ such that that $f^{-n_j}(x) \to x$. Therefore, the center Lyapunov exponent satisfies
\[\chi^c(x)= \lim_{j \to \infty}\frac{1}{n_j}\log \lambda^c_x(n_j)= \lim_{j \to \infty}\frac{1}{n_j}\log \frac{\Phi(f^{n_j}(x))}{\Phi(x)}=0,\]
for $m_U$-a.e. $x \in U$. We deduce that for $m_U$-a.e. $x \in U$, the right hand side in~\eqref{eq 1term}
decreases exponentially fast when $\ell \to \infty$, and the right hand side in~\eqref{iter_relat} converges to some polynomial 
\begin{equation}\label{eq polynomial term}
    \tilde P_x^s(\eta)=\sum_{\ell=1}^{+\infty} \frac{\lambda_{x}^s(-\ell)}{\lambda_x^c(-(\ell-1))}P_{f^{-\ell}(x)}^s(\lambda_x^u(-\ell)\eta). 
\end{equation}
Moreover, by the previous discussion, the coefficients of the polynomials $\{\tilde P_x^s\}$ are uniformly bounded with respect to $x$. 
    
Since the stable templates $\{\mathcal{T}_x^s\}_{x \in M}$ are uniformly Lipschitz continuous and $f$ is volume-preserving, there exist constants $K,K'>0$ such that  
\[\left|\frac{\lambda_x^s(-n)}{\lambda_x^c(-n)}\mathcal{T}_{f^{-n}(x)}^s(\lambda_x^u(-n)\eta)\right|\leq K\left|\frac{\lambda_x^s(-n)\lambda_x^u(-n)}{\lambda_x^c(-n)}\right| \leq \frac{K'}{\left| \lambda^c_x(-n)\right|^{2}}, \]
for every $x \in M$. Again since $f|_U$ is ergodic, for $m$-a.e. $x \in U$, we can choose a subsequence $(n_j)_j$ of integers $n_j \to +\infty$ such that $f^{-n_j}(x) \to x$, and
\[\left|\frac{\lambda_x^s(-n_j)}{\lambda_x^c(-n_j)}\mathcal{T}_{f^{-n_j}(x)}^s(\lambda_x^u(-n_j)\eta)\right| \leq \frac{K'}{\left| \lambda^c_x(-n_j)\right|^{2}}=K'\left( \frac{\Phi(x)}{\Phi(f^{-n_j}(x))} \right)^2 \to K'\]
is uniformly bounded. In particular, for $m_U$-a.e. $x \in U$, we can choose a subsequence $(n'_j)_j$ of $(n_j)_j$ such that 
\begin{equation}\label{template_bdd}
     \lim_{j \to+\infty} \frac{\lambda_x^s(-n'_j)}{\lambda_x^c(-n'_j)}\mathcal{T}_{f^{-n'_j}(x)}^s(\lambda_x^u(-n'_j)\eta)=C_\infty \in \mathbb{R}.
\end{equation}
By~\eqref{iter_relat}-\eqref{template_bdd}, we deduce that 
    \begin{equation*} 
    \mathcal{T}_x^s(\eta)=C_\infty+\tilde P_x^s(\eta),
    \end{equation*}
for $m_U$-a.e. $x \in U$. In fact, as $\mathcal{T}_x^s(0)=\tilde P_x^s(0)=0$, we see that $C_\infty=0$. 
Therefore, $\mathcal{T}_x^s=\tilde P_x^s$ is a polynomial with uniformly bounded coefficients for $m_U$-a.e. $x\in U$. Since the limit of polynomials with uniformly bounded degrees and coefficients is still a polynomial, the stable templates $\{\mathcal{T}_x^s\}_{x \in U}$ are uniformly $C^\infty$ for every $x \in U$.
\end{proof}

\begin{remark}\label{finite_reg_version}
   A finite-regularity version of Proposition~\ref{lemma bootsrap} can also be stated. Specifically, if $f$ is assumed to be merely $C^r$ (with $r \geq 5$) rather than $C^\infty$, Proposition~\ref{propo o good} still guarantees the existence of $C^{r-1}$ adapted charts. By following the same proof as above, we conclude that if $E^s \oplus E^u$ is Lipschitz and $f|_{\bar{U}}$ is $\infty$-bunched, then $E^s \oplus E^u$ is $C^{r-1}$ on $U$.
\end{remark}

\section{A dichotomy: accessibility or integrability}\label{sec_dicho}
Recall that by Proposition~\ref{prop ham17hp25}, for a $C^\infty$ volume-preserving partially hyperbolic diffeomorphism $f\colon M \to M$ on a $3$-manifold, if $f$ is not accessible and $E^s \oplus E^u$ is not integrable, then, up to a finite iterate, there exists an $f$-invariant open accessibility class $U$ such that 
\begin{itemize}
    \item $U$ is a $C^1$ submanifold homeomorphic to $\mathbb T^2 \times (0,1)$;
    \item $f|_U$ is ergodic;
    \item the boundary $\partial U$ consists of two \( f \)-invariant $2$-tori tangent to \( E^s \oplus E^u \).
\end{itemize}
We then prove by contradiction that, if $E^s \oplus E^u$ is Lipschitz, this case cannot happen. 

\begin{proposition}\label{prop dic}
Let $f\colon M \to M$ be a $C^\infty$ volume-preserving partially hyperbolic diffeomorphism with $E^s \oplus E^u$ being Lipschitz. Then either $f$ is accessible or $E^s \oplus E^u$ is integrable. 
\end{proposition}
The proof is divided into several parts. In what follows, we assume that $E^s \oplus E^u$ is Lipschitz but not integrable and show that $f$ is accessible. Otherwise, there exists a non-trivial $f$-invariant open accessibility class $U\subsetneq M$, as described above. We begin by proving that $E^s\oplus E^u$ is $C^1$ and constructing an $f$-invariant vector field $X$ tangent to $E^c|_U$, and study the flow $\phi_t$ generated by $X$. This flow enables us to analyze the center dynamics both in the interior of $U$ and on its boundary $\partial U$. On the boundary, we apply a Livshits-type inequality together with the Ledrappier--Young formula to show that the SRB measure is a volume measure. This implies that the center Lyapunov exponent vanishes, and as a consequence, $f$ is $\infty$-bunched on the closure $\bar{U}$.

We then apply the bootstrap argument (Proposition~\ref{lemma bootsrap}) to upgrade the regularity of $E^s \oplus E^u$ to $C^\infty$. As a result, we obtain the existence of a $C^1$ contact form preserved by $f|_U$, whose Reeb vector field is tangent to $E^c$. This leads to a contradiction via Stokes' theorem, thereby completing the proof.

\subsection{An $f$-invariant vector field}

By considering an iterate if necessary, we always assume that $f$ preserves the orientation of $E^c$.
\begin{lemma} The distribution 
$E^s\oplus E^u$ is $C^1$. 
\end{lemma}
\begin{proof}
By Lemma~\ref{lemma ab dc}, $f$ is dynamically coherent. The lemma then follows from Lemma~\ref{lemma lip boot}. 
\end{proof}

\begin{lemma}\label{lemma inv X}
There is an $f$-invariant continuous vector field $X \subset E^c$ such that $X$ is no-where vanishing on $U$ and vanishes everywhere on $M \setminus U$. 
\end{lemma}
\begin{proof}
As in Section~\ref{sec acc and iso}, let $\alpha$ be a Lipschitz nowhere-vanishing $1$-form with $\ker \alpha = E^s \oplus E^u$. Let $X \subset E^c$ be the vector field on $U$ with $\alpha_x(X_x) = \sqrt{\tilde h(x)}$ for every $x \in U$. Recall that by~\eqref{f_etoile_alpha}, we have $f^\ast \alpha = \rho \alpha$, for some positive Lipschitz function $\rho\colon M \to \RR^+$. Since $E^s \oplus E^u$ is Lipschitz, by Lemma~\ref{lemma coboundary}, $\rho(x)=\sqrt{\tilde h(f(x))}/\sqrt{\tilde h(x)}$ is a coboundary on $U$ for some continuous function $\tilde h\colon U \mapsto \RR^+$. We then have
\[\alpha_{f(x)}(Df_x(X_x))=(f^\ast\alpha)_x(X_x)=\rho(x) \sqrt{\tilde h(x)}=\sqrt{\tilde h(f(x))}=\alpha_{f(x)}(X_{f(x)}). \]
Since $E^s \oplus E^u$ is not only Lipschitz but is actually $C^1$, the $1$-form $\alpha$  is $C^1$, and the function $h\colon U \to \mathbb R$ is the density of the continuous 3-form $\theta =\alpha\wedge d\alpha$. In particular, we have $|h(x)|= \tilde h(x)$ for every $x \in U$. Since $\partial U$ is tangent to $E^s \oplus E^u$, the 3-form $\theta$ vanishes on $\partial U$; we can therefore extend $X$ to be an $f$-invariant continuous vector field on $M$ by letting $X(x)=0$ for every $x \notin U$. 
\end{proof}
Let $\phi_t$ be the flow generated by $X$; in particular, we have $\phi_t\circ f=f\circ \phi_t$ and $\phi_t(U)=U$ for every $t \in \RR$. 

\begin{lemma}\label{lemma leaf dynamic}
For any $x \in U$, if $\W^c_U(x):=\W^c(x) \cap U$ is an $f^n|_U$-invariant center leaf for some $n \in \NN$, then $f^n|_{\W^c_U(x)}=\phi_{t_0}|_{\W^c_U(x)}$ for some $t_0  \in \RR^+$. 
\end{lemma}
\begin{proof}
Note that both $f^n|_{\W^c_U(x)}$ and $\phi_t|_{\W^c_U(x)}$ are orientation preserving diffeomorphisms of the open interval $\W^c_U(x)$ , so there is some $t_0 \in \RR^+$ such that $f^n(x)=\phi_{t_0}(x)$. Since $X$ is nowhere vanishing on $U$, for any $x' \in \W^c_U(x)$, there is some $t' \in \RR$ such that $x'=\phi_{t'}(x)$. Then
\[f^n(x')= f^n\circ \phi_{t'}(x)=\phi_{t'}\circ f^n(x)=\phi_{t'}\circ \phi_{t_0}(x)=\phi_{t_0}(x').\qedhere\]
\end{proof}
Since $X$ is nowhere vanishing on $U$, we have $\lim_{t \to \infty} \phi_t(x) \in \partial U$ for every $x \in U$. Let $S \subset \partial U$ be a connected component of the boundary $\partial U$ such that $\lim_{t \to \infty} \phi_t(x_0) \in S$ for some $x_0 \in U$. 
\begin{lemma}\label{lemma flow to S}
We have    $\lim_{t \to \infty} \phi_t(x) \in S$, for every $x\in U$. 
\end{lemma}
\begin{proof}
Since $TS=E^s\oplus E^u$, the center distribution $E^c$ is transverse to $S$. Let $S'\subset U$ be a surface homeomorphic to $S$ and sufficiently close to $S$; then, the no-where vanishing vector field $X \subset E^c$ is transverse to $S'$. Given that
 $\lim_{t \to \infty} \phi_t(x_0) \in S$ for some $x_0 \in U$, the vector $X(x_0')$ points toward $S$, where $x'_0=\W^c_U(x_0) \cap S'$. Therefore, $X(x)$ points  toward
 $S$ for every $x \in S'$, which implies that $\lim_{t \to \infty} \phi_t(x) \in S$ for every $x\in U$. 
\end{proof}

\subsection{Dynamics on the boundary $S$} 
Let $S \subset \partial U$ be a connected component of the boundary $\partial U$ such that $\lim_{t \to \infty} \phi_t(x) \in S$ for every $x \in U$. By Lemma~\ref{lemma ks06}, $S$ is a $C^\infty$ submanifold tangent to $E^s\oplus E^u$ (in fact, a torus, as discussed at the beginning of Section~\ref{sec_dicho}). Without loss of generality, we may assume $f(S)=S$, so the restriction $f|_S\colon S \to S$ is a $C^\infty$ Anosov diffeomorphism.

We first introduce a Livshits type inequality for $C^2$ Anosov diffeomorphisms. 
\begin{lemma}[{see~\cite[Corollary 2]{lt03}}]\label{lemma livsic inequality}
    Let $\varphi\colon X\to X$ be a $C^2$ Anosov diffeomorphism of a closed manifold $X$, and $A\colon X \to \RR$ a H\"older continuous additive cocycle. Then the following statements are equivalent:
    \begin{enumerate}
        \item  $A^n(p) \geq 0$ for every periodic point $p$ with $f^n p=p$;
        \item $A \geq V\circ f-V$ for some H\"older continuous function $V\colon X \to \RR$;
    \item $\int_X A\, d\mu \geq 0$ for any $\varphi$-invariant measure $\mu$. 
    \end{enumerate}
\end{lemma}
We then introduce a property about SRB measures. Recall that for an Anosov diffeomorphism $\varphi\colon X \to X$, an $\varphi$-invariant measure $\mu$ is an~\emph{SRB measure} if the conditionals of $\mu$ along the leaves of the unstable foliation $\W^u$ are absolutely continuous with respect to the volume measure along these leaves (see for example,~\cite{bdv04} for more details). 
\begin{lemma}\label{lemma srb}
    Let $\varphi \colon X\to X$ be a $C^2$ Anosov diffeomorphism on a closed manifold $X$. Let $\mu$ be an SRB measure. Then 
    \[\int_X \log \mathrm{Jac} (D\varphi)\, d\mu \leq 0.\]
The equality holds if and only if $\mu$ is equivalent to the Riemannian volume. 
\end{lemma}
\begin{proof}
By Ruelle inequality~\cite{rue78} and Ledrappier--Young formula~\cite{ly85}, we always have
    \[\int_X \log \mathrm{Jac} (D \varphi|_{E^u})\, d \mu =h_{\mu}(\varphi) \leq -\int_X \log \mathrm{Jac} (D \varphi|_{E^s})\, d \mu,\]
where $h_\mu(\varphi)$ denotes the metric entropy of $f$ with respect to $\mu$. Therefore, 
   \[\int_X \log\mathrm{Jac} (D\varphi)\, d\mu = h_{\mu}(\varphi) +\int_X \log \mathrm{Jac} (D \varphi|_{E^s})\, d \mu \leq 0.\]
The equality holds if and only if 
\[h_{\mu}(\varphi)  +\int_X \log \mathrm{Jac} (D \varphi|_{E^s})\, d \mu=0,\]
which implies that $\mu$ is absolutely continuous along $\W^s$. Therefore, $\mu$ is a $\varphi$-invariant measure which is absolutely continuous and thus equivalent to the Riemannian volume. 
\end{proof}

We then apply the lemmas above to $S \subset \partial U$ and to $\varphi=f|_S$. 
\begin{lemma}\label{lemma boundary dyna}
    The cocycle $\lambda^c|_S\colon S \to \RR^+$, $S \ni x \mapsto \lambda_x^c=\|Df(x)|_{E^c}\|$  is a coboundary. In particular, $\lambda^c_x(n)$ is uniformly bounded on $S$.  
\end{lemma}
\begin{proof}
For any periodic point $p=f^n(p) \in S$, we have $\lambda^c_p(n) \le 1$. 
Otherwise, for $x \in \mathcal{W}^c(p) \cap U$ sufficiently close to $p$, we would have $\lim_{t \to \infty} \phi_t(x)\notin S$, which contradicts with Lemma~\ref{lemma flow to S}. Since $f$ is volume-preserving, we have
\[\mathrm{Jac}(Df^n|_S(p))=\lambda^s_p(n)\lambda^u_p(n)=\frac{1}{\lambda^c_p(n)} \geq 1.\]
Therefore, by Lemma~\ref{lemma livsic inequality}, we obtain
\[   \int_S\log \mathrm{Jac}(Df|_S)\, d\mu_S\geq 0,\]
for any $f_S$-invariant measure $\mu_S$. 

Let us consider the case when $\mu_S$ is the SRB measure. By Lemma~\ref{lemma srb}, $\mu_S$ is an $f|_S$-invariant measure equivalent to the Riemannian volume. Thus, $\mathrm{Jac}(Df|_S) = \lambda^s \lambda^u$ is a coboundary. Combined with the fact that $f$ is volume-preserving, this implies that $\lambda^c$ is also a coboundary.
\end{proof}
\subsection{Smoothness of $E^s \oplus E^u$}\label{subs_smoothn_open_acc}
We then show that if $E^s\oplus E^u$ is Lipschitz then it is actually $C^\infty$. Recall that we have fixed a non-trivial $f$-invariant open accessibility class $U\subsetneq M$ of $f$. By applying Proposition~\ref{lemma bootsrap}, in order to show that  $E^s\oplus E^u$ is $C^\infty$ on $U$, it suffices to check that $f_{\bar U}$ is $\infty$-bunched. 
\begin{lemma}\label{lemma infty bunched}
The center Lyapunov exponent $\chi^c(\nu)$ vanishes for every $f$-invariant ergodic measure $\nu$ supported on $\bar{U}$. In particular, the map $f_{\bar U}$ is $\infty$-bunched. 
\end{lemma}
\begin{proof}
Let $x \in \bar{U}$ be a $\nu$-typical point. If $x \in U$, then by Corollary~\ref{coro center iso}, the center Lyapunov exponent satisfies
\[
\chi^c(\nu) 
= \lim_{n \to \infty} \frac{1}{n} \log \lambda^c_x(n)
= \lim_{n \to \infty} \frac{1}{n} \log \frac{\sqrt{h(f^n (x))}}{\sqrt{h(x)}} 
= 0.
\]
If $x \in \partial U$, we also have $\chi^c(\nu) = 0$ by Lemma~\ref{lemma boundary dyna}. 

Therefore, by Lemma~\ref{lemma dwx}, for every $\varepsilon > 0$ there exists $n \in \mathbb{N}$ such that
\[
e^{-n\varepsilon} \le \lambda^c_x(n) \le e^{n\varepsilon}, 
\quad \forall\, x \in \bar{U}.
\]
Since $\varepsilon$ can be chosen arbitrarily small, this implies that $f$ is $r$-bunched on $\bar{U}$ for every $r > 0$.
\end{proof}
As a corollary, we have:
\begin{corollary}\label{coro c1 contact form}
The restriction $f|_U$ preserves a $C^1$ contact form $\beta$, with $\ker \beta = E^s \oplus E^u$, and whose Reeb vector field is tangent to $E^c$. 
\end{corollary}

\begin{proof}
Since $\dim M = 3$ and $f$ is volume-preserving, $f$ is $1$-strongly bunched. The corollary then follows from Corollary~\ref{coro ec smooth}. 
\end{proof}
\subsection{Proof of Proposition~\ref{prop dic}}\label{sec stokes} As previously, we assume that $E^s \oplus E^u$ is not integrable, but that $f$ is not accessible, and consider a non-trivial $f$-invariant open accessibility class $U\subsetneq M$ of $f$. We derive a contradiction via Stokes' theorem. 

By Corollary~\ref{coro c1 contact form}, $f|_U$ preserves a $C^1$ contact form $\beta$ whose Reeb vector field $R$ is tangent to $E^c$. Since the boundary $\partial U$ consists of \( f \)-invariant $2$-tori tangent to \( E^s \oplus E^u \), we see that $E^c$ is uniformly transverse to $\partial U$. In particular, we can choose a $C^1$ surface $S \subset U$ homeomorphic to $\mathbb T^2$ and sufficiently close to $\partial U$ such that $E^c$ is transverse to $S$. By Lemma~\ref{lemma non dege}, $d\beta|_{TS}$ is nondegenerate, and hence
\[
\int_S d\beta \neq 0,
\]
which contradicts Stokes' theorem.

\section{Proof of the main results}\label{sec_proofs}
\subsection{Proof of Theorem~\ref{theorem main}}
For~\eqref{t1 lambda c}, Proposition~\ref{prop dic} implies that either $f$ is accessible or $E^s \oplus E^u$ is integrable. Note that $\lambda^c_x$ is a coboundary if and only if $f$ is center isometric, i.e., there is a continuous metric such that $\lambda^c_x=1$, for all $x \in M$. In the accessible case, $\lambda^c_x$ is a coboundary by Corollary~\ref{coro center iso}. 
If $E^s \oplus E^u$ is integrable, then by Proposition~\ref{prop ham17hp25} the map $f$ is either center isometric or topologically conjugate to an Anosov automorphism of $\TT^3$. In the latter case $f$ is Anosov, and $\lambda^c$ is cohomologous to a constant by~\cite[Theorem 1.1]{gs20}.

For~\eqref{t1 dc}, center isometric partially hyperbolic systems are dynamically coherent by~\cite[Theorem 7.5]{hhucoherence}, while Anosov diffeomorphisms on $\TT^3$ are always dynamically coherent~\cite{pot14}.

For~\eqref{t1 boot}, if $f$ is center isometric and accessible, then $E^s \oplus E^u$ is $C^\infty$ by Proposition~\ref{lemma bootsrap}. If $E^s \oplus E^u$ integrates to a foliation $\W^{su}$, then the leaves of $\W^{su}$ are $C^\infty$ by Lemma~\ref{lemma ks06}. Since $f$ is $\infty$-bunched, the holonomy of $\W^{su}$ along $\W^c$ is uniformly $C^\infty$. Journé's lemma~\cite{jou88} then yields that $\W^{su}$ is a $C^\infty$ foliation, and hence $E^s \oplus E^u$ is $C^\infty$. If $f$ is Anosov, smoothness of $E^s \oplus E^u$ follows from~\cite[Theorem 1.3]{ks25}. 

Moreover, since $E^s\oplus E^u$ is $C^\infty$, the projectivization $N:=\PP E^{su}$ is a $C^\infty$ fiber bundle over $M$. Let $F\colon N \to N$ be the bundle map induced by $Df|_{E^s\oplus E^u}$. Then $\PP E^s$ and $\PP E^u$ are $F$-invariant sections of $N$. By the $C^r$ section Theorem~\cite{hps77}, both sections are $C^{1+\alpha}$ for some $\alpha>0$, and hence $E^s$ and $E^u$ are also $C^{1+\alpha}$.

To complete the proof of Theorem~\ref{theorem main}, it remains to establish the rigidity result in the accessible case. In this setting $f$ is ergodic, hence topologically transitive. By~\eqref{t1 lambda c} and~\eqref{t1 boot} in Theorem~\ref{theorem main}, $f$ is center isometric and $E^s \oplus E^u$ is $C^\infty$. Corollary~\ref{coro ec smooth} then yields that $E^c$ is $C^\infty$ and $f$ preserves a $C^\infty$ contact form. The $C^\infty$ classification follows from the next proposition.

\begin{proposition}[{see~\cite[Theorem C]{bz20},~\cite[Theorems A and F]{avw22}}]\label{prop avw22}
Let $f \colon M \to M$ be a $C^\infty$ volume-preserving partially hyperbolic diffeomorphism on a closed $3$-manifold. Assume that $f$ is center isometric and topologically transitive, and that $\W^c$ is absolutely continuous. Then up to finite lifts and iterates, $f$ is $C^\infty$-conjugate to one of the following:
\begin{itemize}
    \item an isometric extension of a volume-preserving Anosov diffeomorphism on $\TT^2$; 
    \item the time-one map of a $C^\infty$ Anosov flow.
\end{itemize}
\end{proposition}
Finally, if $f$ is $C^\infty$-conjugate to an isometric extension of an Anosov diffeomorphism, then $M$ is either a torus or a Heisenberg manifold. We prove by contradiction that $M \neq \TT^3$. Suppose instead that $M = \TT^3$, then $\W^c$ is a trivial bundle over $\TT^2$. Let $S \subset \TT^3$ be a global section transverse to $\W^c$, and let $\beta$ be the $C^\infty$ contact form preserved by $f$. Arguing as in Section~\ref{sec stokes}, Lemma~\ref{lemma non dege} implies that $d\beta|_{TS}$ is nondegenerate. Consequently,
\[
\int_S d\beta \neq 0,
\]
which contradicts Stokes' theorem.

\subsection{Proof of Corollary~\ref{theorem QNI}}
For $\ell \geq 1$, we first recall the definition of $\ell$-integrability. 
\begin{definition}[{see~\cite[Definition 1.1]{epz}}]\label{def JI l}
Let $f\colon M \to M$ be a $C^\infty$ partially hyperbolic diffeomorphism on a closed $3$-manifold. We say that a compact invariant set $\Lambda$ is~\emph{$\ell$-integrable} if there exists a continuous family of $C^\ell$ smooth surfaces $\{S_x\}_{x \in \Lambda}$ such that:
\begin{itemize}
    \item $\W^u_{\mathrm{loc}}(x) \cup \W^s_{\mathrm{loc}}(x)\subset S_x$;
    \item for every $x \in \Lambda$ and every $y \in \W^u_{\mathrm{loc}}(x)\cap \Lambda$ (resp.\ $y \in \W^s_{\mathrm{loc}}(x)\cap \Lambda$), the curve $\W^s_{\mathrm{loc}}(y)$ is tangent up to order $\ell$ to $S_x$ at $y$ (resp.\ $\W^u_{\mathrm{loc}}(y)$ is tangent up to order $\ell$ to $S_x$ at $y$).
\end{itemize}
In particular, we say that $f$ is $\ell$-integrable if the ambient manifold $M$ is $\ell$-integrable. 
\end{definition}
Here, when we say that a curve $\gamma$ is tangent to order $\ell$ to $S_x$ at $y$, we mean that there exists a constant $C > 0$ such that, when parametrized by arc length, the distance from a point $z \in \gamma$ to the surface $S_x$ is less than $C t^\ell$, where $t$ denotes the arc length from $z$ to $y$. 

Since we do not assume that $f$ is dynamically coherent, we first recall the definition of~\emph{weak dynamical coherence}~\cite{bbi04} as an alternative. A $C^0$ distribution $E$ on a manifold $M$ is called~\emph{weakly integrable} if for each point $x$ there exists an immersed, complete $C^1$ manifold $W(x)$ containing $x$ and everywhere tangent to $E$. We refer to $W(x)$ as an~\emph{integral manifold} of $E$. Note that, a priori, the integral manifolds $W(x)$ may be self-intersecting and may not form a foliation of $M$. In particular, if the invariant distributions $E^c$, $E^{cs}$, and $E^{cu}$ of a partially hyperbolic diffeomorphism $f$ are weakly integrable, we call $f$~\emph{weakly dynamically coherent}.

\begin{lemma}[{see~\cite[Proposition 3.4]{bbi04}}]
Let $f \colon M \to M$ be a $C^1$ partially hyperbolic diffeomorphism of a compact manifold $M$ with $\dim E^c = 1$. Then $f$ is~\emph{weakly dynamically coherent}.
\end{lemma}

For weakly dynamically coherent partially hyperbolic diffeomorphisms, we denote by $\W^\ast(x)$ the $C^1$ manifold tangent to $E^\ast$ for $\ast = c, cs, cu$, in order to distinguish them from the true foliations $\W^\sigma$, $\sigma = s, u$.

\begin{lemma}[{see~\cite[Proposition B.7]{hhu1d}}]\label{lemma hhu1d}
Let $f \colon M \to M$ be a $C^{2}$ partially hyperbolic diffeomorphism with $\dim E^c = 1$. Then the stable (resp. unstable) foliation restricted to $\W^{cs}_{\mathrm{loc}}(x)$ (resp. $\W^{cu}_{\mathrm{loc}}(x)$) is uniformly $C^1$ for every $x \in M$.
\end{lemma}

\begin{lemma}\label{lemma lip JI l}
Let $\Lambda$ be a compact $f$-invariant set. If $\Lambda$ is $\ell$-integrable for some $\ell \ge 2$, then $E^s \oplus E^u$ is $C^1$ on $\Lambda$. 
\end{lemma}

\begin{proof}
We first prove that $E^s \oplus E^u$ is uniformly $C^1$ along $\W^u$; the proof for $\W^s$ is similar. Since $\Lambda$ is $\ell$-integrable, for every $x \in \Lambda$ and $y \in \W^u_{\mathrm{loc}}(x) \cap \Lambda$, we have $E^u(y) \subset T_y S_x$ for some $C^\ell$ surface $S_x$ (see Definition~\ref{def JI l}). We claim that $E^s(y) \subset T_y S_x$. Otherwise, the curve $\W^s_{\mathrm{loc}}(y)$ cannot be tangent to order $\ell$ to $S_x$ at $y$ for any $\ell > 1$, which contradicts our assumption. Therefore,
\[
(E^s \oplus E^u)(y) = T_y S_x.
\]
Since $\{S_x\}$ is a continuous family of $C^\ell$ surfaces with $\ell \ge 2$, this implies that $E^s \oplus E^u$ is uniformly $C^1$ along $\W^u$.

By Lemma~\ref{lemma hhu1d}, the bundle $E^s \oplus E^u$ is uniformly $C^1$ along each center leaf $\W^c_{\mathrm{loc}}(x)$ for every $x \in M$. Since the foliations $\W^s_{\mathrm{loc}}$, $\W^u_{\mathrm{loc}}$, and $\W^c_{\mathrm{loc}}$ form a local product structure near every point and vary continuously, it follows from a standard calculus argument that $E^s \oplus E^u$ is $C^1$ on $\Lambda$.
\end{proof}

We now present the proof of Corollary~\ref{theorem QNI}. 

If $E^s \oplus E^u$ is integrable, then by Lemma~\ref{lemma ks06} there exists an integral foliation $\W^{su}$ with uniformly $C^\infty$ leaves, and $f$ is $\ell$-integrable for any $\ell$. 

Conversely, assume that \( f \) is \( \ell \)-integrable for some \( \ell > 2 \).
Then, by Lemma~\ref{lemma lip JI l}, \( E^s \oplus E^u \) is Lipschitz.
Since \( f \) is volume-preserving, Theorem~\ref{theorem main} implies that either \( E^s \oplus E^u \) is integrable, or \( f \) preserves a \( C^\infty \) contact form.
The following lemma rules out the second case, thereby completing the proof of the first point of Corollary~\ref{theorem QNI}.
\begin{lemma}\label{lemm_qni_contact}
    If $f$ preserves a $C^1$ contact form $\alpha$ with $\ker \alpha =E^s\oplus E^u$, then $f$ is not $\ell$-integrable for any $\ell>2$. 
\end{lemma}

\begin{proof}
Assume on the contrary that $f$ is $\ell$-integrable for some $\ell>2$. For any $x \in M$ and sufficiently small $\epsilon>0$, we take points $y,z,w_1,w_2 \in M$ such that
    \begin{itemize}
        \item $y \in \W^s(x)$, $z \in \W^u(x)$, $w_1 \in \W^u(y)$, $w_2 \in \W^s(z)\cap \W^c(w_1)$,
        \item $d_{\W^s}(x,y), d_{\W^u}(y,w_1), d_{\W^u}(x,z), d_{\W^s}(z,w_2) \in (\epsilon,2\epsilon)$, where $d_{\W^\ast}$ denotes the leafwise distance, $*=s,u$. 
    \end{itemize}

Let $S_x$ be the $C^\ell$ surface as in Definition~\ref{def JI l}. Then we have $y,z \in S_x$. Moreover, let $w' := \W^c_\text{loc}(w_1) \cap S_x = \W^c_\text{loc}(w_2) \cap S_x$; then, we have
\[
d(w_1,w_2) \le d(w_1,w') + d(w',w_2) \le C\epsilon^\ell
\]
for some constant $C>0$. 

On the other hand, let $\Omega$ be a surface transverse to $E^c$ on its interior and whose boundary consists of stable and unstable plaques joining $x,y,z,w_1,w_2$ and the geodesic segment $\gamma$ joining $w_1$ and $w_2$. Since $\ker \alpha =E^s\oplus E^u$, the integration of $\alpha$ along  stable/unstable arcs (contained in leaves of $\W^s$ or $\W^u$) vanishes. By Stokes' formula, we have

\[
\int_\gamma \alpha = \int_{\partial \Omega}\alpha = \int_\Omega d\alpha \ge C'\epsilon^2,
\]  
for some $C'>0$. This gives a contradiction since $\alpha$ is bounded and $|\gamma| = d(w_1,w_2) < C\epsilon^\ell$, with $\ell>2$. 
\end{proof}

We now consider the second point of Corollary~\ref{theorem QNI}.
Assume that \( E^s \oplus E^u \) is not globally jointly integrable.
In particular, we fall into either case~\eqref{cas_accessib} or case~\eqref{cas_trois} of Proposition~\ref{prop ham17hp25}.

In case~\eqref{cas_accessib}, the whole manifold \( M \) is a single accessibility class.
Consequently, by the first point of Corollary~\ref{theorem QNI}, \( M \) is not \( \ell \)-integrable for any \( \ell > 2 \).

Assume we are now in case~\eqref{cas_trois} of Proposition~\ref{prop ham17hp25}.
Then, up to passing to a finite iterate and finite covers, the following holds for some integer \( n \geq 1 \). 
Let \( A \subset M \) be an accessibility class. Then,
\begin{enumerate}
    \item either \( E^s \oplus E^u|_{A} \) is integrable, and \( A \) is an \( f^n \)-invariant \( 2 \)-torus tangent to \( E^s \oplus E^u \);
    \item or \( A \) is an \( f^n \)-invariant open accessibility class.
\end{enumerate}
In the latter case, it remains to prove that \( \overline{A} \) is not \( \ell \)-integrable for any \( \ell > 2 \).
Let \( U \) denote the \( f \)-invariant set \( U := \bigcup_{k=0}^{n-1} f^k(A) \), and let \( \Lambda := \overline{U} \).
Assume, by contradiction, that the compact set \( \Lambda \) is \( \ell \)-integrable for some \( \ell > 2 \).
By Lemma~\ref{lemma lip JI l}, \( E^s \oplus E^u \) would then be Lipschitz on \( \Lambda \).
Following the argument in Subsection~\ref{subs_smoothn_open_acc}, Corollary~\ref{coro c1 contact form} would imply that \( f|_U \) preserves a \( C^1 \) contact form \( \beta \) with \( \ker \beta = E^s \oplus E^u \).
However, by Lemma~\ref{lemm_qni_contact}, this leads to a contradiction, thereby completing the proof of Corollary~\ref{theorem QNI}.

\subsection{Proof of Corollary~\ref{theorem measure rigid}}

Assume that \( E^s \oplus E^u \) is not integrable.
In particular, this corresponds to either case~\eqref{cas_accessib} or case~\eqref{cas_trois} of Proposition~\ref{prop ham17hp25}. 
Fix an ergodic \( u \)-Gibbs state \( \mu \) with a positive center Lyapunov exponent and \( m(\mathrm{supp}(\mu)) > 0 \).

First, consider case~\eqref{cas_accessib}, i.e., when \( f \) is accessible.
In this case, \( f \) is ergodic with respect to \( m \), and since \( m(\mathrm{supp}(\mu)) > 0 \), it follows that \( \mathrm{supp}(\mu) = M \).
As \( E^s \oplus E^u \) is not integrable, Corollary~\ref{theorem QNI} implies that \( f \) is not \( \ell \)-integrable for any \( \ell > 2 \), and by~\cite[Theorem 8.1]{epz}, \( \mu \) satisfies the QNI condition.
Furthermore,~\cite[Theorem 1.2]{epz} ensures that \( \mu \) is physical, meaning its basin has positive \( m \)-measure (hence full measure, by ergodicity of the volume measure $m$).  
Applying Birkhoff's ergodic theorem to a point in the intersection of the basins of \( m \) and \( \mu \), we conclude that \( \mu = m \).

We now consider case~\eqref{cas_trois} of Proposition~\ref{prop ham17hp25}.
Since \( m(\mathrm{supp}(\mu)) > 0 \), Lemma~\ref{lemma ham measure} implies that \( \mathrm{supp}(\mu) = \overline{U} \), where \( U := \bigcup_{k=0}^{n-1} f^k(U_0) \) for some \( f^n \)-invariant open accessibility class \( U_0 \) and integer \( n \geq 1 \).  
Let us consider the $f$-invariant compact set $\Lambda:=\overline{U}=\mathrm{supp}(\mu)$. 
By point~\eqref{point_deux_corb} of Corollary~\ref{theorem QNI}, \( \mathrm{supp}(\mu) =\Lambda \) is not \( \ell \)-integrable for any \( \ell > 2 \);  it follows from~\cite[Theorem 8.1]{epz} that \( \mu \) satisfies the QNI condition.
Furthermore,~\cite[Theorem 1.2]{epz} implies that \( \mu \) is physical.
As in case~\eqref{cas_accessib}, Birkhoff's ergodic theorem yields \( \mu = m_U \), where \( m_U := \frac{m|_U}{m(U)} \).
This completes the proof of the dichotomy in Corollary~\ref{theorem measure rigid}. 

The second part of Corollary~\ref{theorem measure rigid} follows from the first part, combined with the following two results (Propositions~\ref{propo_ergo_vol} and~\ref{lemma minimal} below). Indeed, if the invariant volume $m$ is ergodic and has a non-zero center Lyapunov exponent, and $\mathcal{W}^u$ is minimal, then for any ergodic $u$-Gibbs state $\mu$ with a nonzero center Lyapunov exponent $\chi^c\neq 0$,
\begin{itemize}
    \item either $\chi^c>0$, and then $\mu=m$, by the first part of Corollary~\ref{theorem measure rigid};
    \item or $\chi^c<0$, in which case $\mu$ is physical (by the Hopf argument); by the ergodicity of $\mu$ and $m$, taking a point in the intersection of their basins, we also conclude that $\mu=m$. 
\end{itemize} 
\begin{proposition}[{see~\cite[Main Theorem]{hhu1d},\cite[Corollary 0.3]{bb03}}]\label{propo_ergo_vol}
There exists a $C^1$-open and dense subset of volume-preserving partially hyperbolic diffeomorphisms on $3$-manifolds such that $f$ is ergodic with respect to the invariant volume and has non-zero center Lyapunov exponent. 
\end{proposition}

\begin{proposition}[{see~\cite{avila2025minimalitystrongfoliationsanosov,CP}}]\label{lemma minimal}
There exists a $C^1$-open and dense subset of volume-preserving partially hyperbolic diffeomorphisms on $3$-manifolds such that both $\W^s$ and $\W^u$ are minimal.
\end{proposition}
\begin{remark}\label{remark minimal}
In the work of Avila--Crovisier--Wilkinson~\cite{avila2025minimalitystrongfoliationsanosov}, it is shown that under a weak expansion condition on the center bundle (called~\emph{some hyperbolicity}, or ``SH''), any $s$-transverse partially hyperbolic lamination contains a disk tangent to the center-unstable direction. As an application, the authors prove the minimality of the (strong) unstable foliation for Anosov diffeomorphisms on $\mathbb{T}^3$ (see also~\cite{acepwz}). In a forthcoming work of Crovisier--Potrie~\cite{CP}, the SH property is also used to study the minimality of the (strong) unstable foliation. In particular, the authors show that for a $C^1$-open and dense set of volume-preserving partially hyperbolic systems with one-dimensional center, either the system is Anosov or both the stable and unstable foliations are minimal. 
\end{remark}

\subsection{Proof of Corollary~\ref{theorem smooth}}
By Theorem~\ref{theorem main} together with the $C^0$ classification for volume-preserving partially hyperbolic diffeomorphisms with integrable $E^s \oplus E^u$~\cite{fp25,ham17} (see Proposition~\ref{prop ham17hp25}), we obtain the following topological classification.
\begin{proposition}\label{prop topo}
Let $f\colon M \to M$ be a $C^\infty$ volume-preserving partially hyperbolic diffeomorphism on a closed $3$-manifold. If $E^s \oplus E^u$ is Lipschitz, then up to a finite cover and iteration, $f$ is $C^0$-conjugate to one of the following:
\begin{itemize}
    \item an Anosov automorphism of $\mathbb{T}^3$;
    \item an isometric extension of an Anosov diffeomorphism on $\mathbb{T}^2$;
    \item the time-one map of an Anosov flow.
\end{itemize}
\end{proposition}
\begin{remark}
Proposition~\ref{prop topo} is also of interest on its own. Without additional regularity assumptions on $E^s \oplus E^u$, one cannot expect a classification up to topological conjugacy in general. For example, time changes of an Anosov flow yield time-one maps that are leaf conjugate but, in general, not topologically conjugate.
\end{remark}
Assuming in addition that $\W^c$ is absolutely continuous, we can upgrade the conjugacy from $C^0$ to $C^\infty$.

In the first case, Theorem~\ref{theorem main} implies that $f$ is Anosov and that the center cocycle $x \mapsto\lambda^c_x$ is cohomologous to a constant. Replacing $f$ by $f^{-1}$ if necessary, we may assume that $f$ is uniformly expanding along $\W^c$. Since $\W^c$ is absolutely continuous, it follows from~\cite[Proposition~2]{gog12} that the unstable cocycle $x \mapsto\lambda^u_x$ is cohomologous to a constant. As $f$ preserves volume, the multiplicative cocycle $x \mapsto \lambda^s_x \lambda^u_x \lambda^c_x$ is a coboundary, and hence so is the stable cocycle $x \mapsto\lambda^s_x$. Therefore, by~\cite[Theorem~1.1]{dg24}, $f$ is $C^\infty$-conjugate to its linear part.

In the remaining cases, Proposition~\ref{prop avw22} implies that $f$ is $C^\infty$-conjugate to the corresponding model. Moreover, if $f$ is $C^\infty$-conjugate to the time-one map of an Anosov flow $\varphi_t$ with $E^s \oplus E^u$ of class $C^1$, then $\varphi_t$ is either a contact flow or a suspension flow~\cite{pla72} (see also~\cite[Theorem~3]{fh03}).

\subsection{Proof of Corollary~\ref{theorem contact}}
\begin{lemma}
We have $R(x) \in E^c(x)$ for all $x \in M$, and $\ker \alpha = E^s \oplus E^u$, where $R \subset TM$ is the Reeb vector field associated with the contact form $\alpha$.
\end{lemma}

\begin{proof}
Since $\alpha$ is $f$-invariant, the Reeb vector field $R$ is $f$-invariant and bounded, which implies that $R(x) \in E^c(x)$, for all $x \in M$. For any $v \in E^s(x)$, we have
\[
\alpha(v) = ((f^n)^\ast \alpha)(v) = \alpha(Df^n(v)) \to 0
\]
as $n \to \infty$, which implies $E^s \subset \ker \alpha$. Similarly, $E^u \subset \ker \alpha$, and hence $\ker \alpha = E^s \oplus E^u$.
\end{proof}

To apply Theorem~\ref{theorem main}, it remains to show that $f$ preserves a $C^\infty$ volume form.

\begin{lemma}
The diffeomorphism $f$ is accessible.
\end{lemma}

\begin{proof}
We show that every accessibility class $A$ is open. Since $M$ is connected and accessibility classes are pairwise disjoint, this implies that $f$ is accessible.

We argue by contradiction. By~\cite[Proposition 5]{bhhtu08}, if $A$ is not open, then $f$ is jointly integrable at every point $x \in A$, i.e., there exists $\epsilon > 0$ such that for all $y \in \W^u_\epsilon(x)$ and $z \in \W^s_\epsilon(x)$, we have
\[
\W^s_{\mathrm{loc}}(y) \cap \W^u_{\mathrm{loc}}(z) \neq \emptyset.
\]
In particular, choose $w \in \W^s_{\mathrm{loc}}(y) \cap \W^u_{\mathrm{loc}}(z)$, and let $\Omega$ be a surface transverse to $E^c$ whose boundary consists of stable and unstable plaques joining $x, y, z, w$. Since $\ker \alpha = E^s \oplus E^u$, the integral of $\alpha$ along stable and unstable arcs (contained in leaves of $\W^s$ or $\W^u$) vanishes. Since $d\alpha$ is non-degenerate on $\Omega$, Stokes' theorem gives
\[
0 \neq \int_\Omega d\alpha = \int_{\partial \Omega} \alpha = 0,
\]
a contradiction.
\end{proof}

\begin{lemma}
The diffeomorphism $f$ preserves a $C^\infty$ volume form $\theta$.
\end{lemma}

\begin{proof}
Let $\theta = \alpha \wedge d\alpha$, then $\theta$ is a continuous $f$-invariant $3$-form. Let $m$ be the Riemannian volume form and let $\rho$ be the density function such that $\theta = \rho m$. The invariance $f^\ast \theta = \theta$ implies
\[
-\log \det(Df) = \log \rho \circ f - \log \rho.
\]
Since $f$ preserves the Reeb vector field $R$, it is a center isometry and, in particular, it is strongly $\infty$-bunched. By Proposition~\ref{prop wil13}, $\rho$ is $C^\infty$, and hence $\theta$ is a $C^\infty$ volume form.
\end{proof}

Since $\alpha$ is $C^1$, the distribution $\ker \alpha = E^s \oplus E^u$ is also $C^1$. Corollary~\ref{theorem contact} now follows from Theorem~\ref{theorem main}.

\subsection{Proof of Corollary~\ref{theorem c2}}

By Theorem~\ref{theorem main}, it suffices to consider the case where $f$ is either the time-one map of an Anosov flow or an isometric extension. In the flow case, the conclusion follows immediately from~\cite[Theorem 4.6]{ghy93} and~\cite[Theorem 1.1]{ghy87}.

We now consider the case where $f$ is an isometric extension. Since $\W^c$ is smooth, the quotient
\[
\hat f \colon M/\W^c \to M/\W^c
\]
is a $C^\infty$ Anosov diffeomorphism. Moreover, the induced bundles $\hat E^s$ and $\hat E^u$ are $C^{1+\mathrm{Lip}}$. By~\cite[Corollary 3.3]{ghy93}, the map $\hat f$ is $C^\infty$ conjugate to an affine automorphism, and $\hat E^s$ and $\hat E^u$ are $C^\infty$.

Since $E^s \oplus E^u$ is $C^\infty$ and transverse to $E^c$, the projection $\pi \colon M \to M/\W^c$ induces a $C^\infty$ vector bundle isomorphism
\[
d\pi|_{E^s\oplus E^u} \colon E^s \oplus E^u \to T(M/\W^c).
\]
Let $\tilde E^s := (d\pi|_{E^s\oplus E^u} )^{-1}(\hat E^s)$, then $\tilde E^s$ is a $C^\infty$ subbundle of $TM$, and it is invariant under $Df$ and exponentially contracted. By uniqueness of the stable bundle, it follows that $E^s = \tilde E^s$, hence $E^s$ is $C^\infty$. The same argument applies to $E^u$.

Finally, since $f$ is isometric along the center direction and descends to an affine automorphism on $\TT^2$,we can choose a $C^\infty$ global frame $X_\ast \subset E^\ast$, $\ast = s,c,u$, such that $Df(x)$ is represented by a constant matrix for all $x \in \TT^3$. The conclusion then follows from~\cite[Theorem 1.1]{am23}.





\section{An example on a Heisenberg manifold}\label{sec_example}
Let $f\colon M \to M$ be a $C^\infty$ volume-preserving partially hyperbolic diffeomorphism on a closed $3$-manifold $M$. By Theorem~\ref{theorem main}, if $f$ is accessible and the distributions $E^s \oplus E^u$ and $E^c$ are $C^\infty$, then, up to a finite cover and an iterate, $f$ is $C^\infty$-conjugate to one of the following models:
\begin{itemize}
    \item an isometric extension of a volume-preserving Anosov diffeomorphism on $\TT^2$, with total space $M$ a Heisenberg nilmanifold;
    \item the time-one map of a contact Anosov flow.
\end{itemize}

A natural question is whether one can obtain further rigidity, namely whether $E^s$ and $E^u$ are necessarily $C^\infty$, and whether $f$ is $C^\infty$-conjugate to an algebraic model as in Corollary~\ref{theorem c2}. 

In this section, we show that this is in general impossible: the classification results in Theorem~\ref{theorem main} and Corollary~\ref{theorem c2} are sharp. More precisely, without additional assumptions, diffeomorphisms arising in the classification of Theorem~\ref{theorem main} need not be even topologically conjugate to the corresponding algebraic models.

For contact Anosov flows, Foulon--Hasselblatt~\cite{fhflow} constructed examples that are not topologically orbit equivalent to algebraic flows. We therefore restrict our attention to partially hyperbolic systems on Heisenberg manifolds.
\begin{example}\label{example contact}
Let $M=\mathbb{R}^3/\sim$ be the $3$-dimensional Heisenberg manifold, where
\[
(x,y,z)\sim (x+m,\, y+n,\, z+k+my+\tfrac{1}{2}mn), \quad m,n,k\in\mathbb{Z}.
\]
Let $\alpha=dz-x\,dy$ be the standard contact form on $M$. Define
\[
L\colon (x,y,z)\mapsto \left(2x+y,\, x+y,\, z+x^2+xy+\tfrac{1}{2}y^2\right).
\]
A direct computation shows that $L^\ast \alpha=\alpha$.

For $\varepsilon\in\mathbb{R}$ sufficiently small, define
\[
u_\varepsilon\colon x \mapsto \varepsilon\sin(2\pi x), \qquad
U_\varepsilon\colon x\mapsto\varepsilon x\sin(2\pi x)+\frac{\varepsilon}{2\pi}\bigl(\cos(2\pi x)-1\bigr),
\]
and set
\[
H_\varepsilon\colon (x,y,z)\mapsto \bigl(x,\, y+u_\varepsilon(x),\, z+U_\varepsilon(x)\bigr).
\]
Since $U'_\varepsilon(x)=xu'_\varepsilon(x)$, we have
\[
H_\varepsilon^\ast\alpha
= d(z+U_\varepsilon)-x\,d(y+u_\varepsilon)
= dz - x\,dy
= \alpha.
\]
Thus $H_\varepsilon$ preserves the contact form.

We now verify that $H_\varepsilon$ is well-defined on $M$. Using the identities
\[
u_\varepsilon(x+1)=u_\varepsilon(x), \qquad U_\varepsilon(x+1)-U_\varepsilon(x)=u_\varepsilon(x),
\]
by a direct calculation, we can check that $H_\varepsilon$ commutes with the deck transformations defining the quotient. Therefore, $H_\varepsilon$ descends to a well-defined diffeomorphism on $M$.

Define
\[
F_\varepsilon := L \circ H_\varepsilon.
\]
Then $F_\varepsilon$ is partially hyperbolic for every sufficiently small $\epsilon>0$. 
\end{example}

\begin{lemma}
\label{lemma c2-}
The diffeomorphism $F_\varepsilon$ belongs to the class described in Theorem~\ref{theorem main} for all sufficiently small $\varepsilon>0$. In particular, it preserves volume, and the distributions $E^s \oplus E^u$ and $E^c$ are $C^\infty$. Moreover, for any $\eta>0$, there exists $\varepsilon(\eta)>0$ such that for all $0<\varepsilon<\varepsilon(\eta)$, the distributions $E^s$ and $E^u$ are of class $C^{2-\eta}$.
\end{lemma}
\begin{proof}
By construction, $F_\varepsilon := L \circ H_\varepsilon$ preserves the contact form $\alpha$, and hence the volume form $\alpha \wedge d\alpha$. It follows from Corollary~\ref{theorem contact} that $E^s \oplus E^u$ and $E^c$ are $C^\infty$.

For the regularity of $E^s$ and $E^u$, note that $F_\varepsilon$ is a $C^\infty$ perturbation of the linear map $L$, so the strength of domination can be made arbitrarily close to that of $L$ by taking $\varepsilon$ sufficiently small. In particular, the $C^{2-\eta}$ regularity of $E^s$ and $E^u$ follows from the $C^r$ section theorem~\cite{hps77}.
\end{proof}
\begin{lemma}
For any $\delta>0$, there exists $0<\varepsilon<\delta$ such that $F_\varepsilon$ is not $C^0$-conjugate to any affine automorphism of $M$.
\end{lemma}
\begin{proof}
The idea is to compare the return maps on periodic center fibers. For the affine models, the first return map on any periodic center fiber is a circle rotation with rational rotation number. On the other hand, along a suitable periodic center fiber of $F_\varepsilon$, the rotation number of the return map depends $C^1$-smoothly on $\varepsilon$ and varies non-trivially at $\varepsilon=0$. In particular, the return map has irrational rotation number for some sufficiently small $\varepsilon>0$. This rules out a topological conjugacy to an affine automorphism. 

For any $\varepsilon>0$, write,
\[
F_\varepsilon\colon (x,y,z)\mapsto \bigl(g_\varepsilon(x,y),\, z+\tau_\varepsilon(x,y)\bigr),
\]
where
\[
g_\varepsilon \colon (x,y) \mapsto \bigl(2x+y+u_\varepsilon(x),\,x+y+u_\varepsilon(x)\bigr)
\]
and
\[
\tau_\varepsilon\colon (x,y)\mapsto  
U_\varepsilon(x)+x^2+x\bigl(y+u_\varepsilon(x)\bigr)
+\frac12\bigl(y+u_\varepsilon(x)\bigr)^2.
\]
At $\varepsilon=0$, the base map is the hyperbolic toral automorphism
\[
g_0\colon (x,y) \mapsto A(x,y):=(2x+y,\,x+y).
\]

Now consider the $2$-periodic orbit of $A$ given by
\[
p_0=\Bigl(\frac15,\frac25\Bigr),\qquad q_0=\Bigl(\frac45,\frac35\Bigr).
\]
Since $A^2-I$ is invertible, by the implicit function theorem there exists a unique $C^1$ family
of $2$-periodic points $p_\varepsilon$ of $g_\varepsilon$ with $p_\varepsilon\to p_0$ as $\varepsilon\to 0$. Let
\[
q_\varepsilon:=g_\varepsilon(p_\varepsilon).
\]
Then the first return map of $F_\varepsilon^2$ on the corresponding periodic center fiber is the circle rotation with angle
\[
R(\varepsilon):=\tau_\varepsilon(p_\varepsilon)+\tau_\varepsilon(q_\varepsilon)\in \mathbb R/\mathbb Z.
\]
Since $(\varepsilon,x,y)\mapsto \tau_\varepsilon(x,y)$ is smooth and $\varepsilon\mapsto p_\varepsilon,q_\varepsilon$ are $C^1$, it follows that $R(\varepsilon)$ is $C^1$.

To compute $R'(0)$, note that
\[
\tau_0(x,y)=x^2+xy+\frac12 y^2,
\qquad
\nabla\tau_0(x,y)=(2x+y,\ x+y),
\]
and
\[
\partial_\varepsilon\tau_\varepsilon(x,y)\big|_{\varepsilon=0}
=
(2x+y)\sin(2\pi x)+\frac{1}{2\pi}\bigl(\cos(2\pi x)-1\bigr).
\]
Differentiating the identity
\[
g_\varepsilon^2(p_\varepsilon)=p_\varepsilon
\]
at $\varepsilon=0$, one gets
\begin{equation}\label{eq:p0prime}
(A^2-I)p_0'=-\partial_\varepsilon(g_\varepsilon^2)(p_0)\big|_{\varepsilon=0}.
\end{equation}
Using
\[
\partial_\varepsilon g_\varepsilon(x,y)\big|_{\varepsilon=0}
=
\sin(2\pi x)(1,1),
\]
we get
\[
\partial_\varepsilon(g_\varepsilon^2)(p_0)\big|_{\varepsilon=0}
=
A\Bigl(\sin\frac{2\pi}{5},\sin\frac{2\pi}{5}\Bigr)
+
\Bigl(\sin\frac{8\pi}{5},\sin\frac{8\pi}{5}\Bigr)
=
\Bigl(2\sin\frac{2\pi}{5},\sin\frac{2\pi}{5}\Bigr).
\]
Solving~\eqref{eq:p0prime}, we obtain
\[
p_0'=\left(-\frac15\sin\frac{2\pi}{5},-\frac25\sin\frac{2\pi}{5}\right).
\]
Since \(q_\varepsilon=g_\varepsilon(p_\varepsilon)\), differentiating at \(\varepsilon=0\) gives
\[
q_0'=Dg_0(p_0)\,p_0'+\partial_\varepsilon g_\varepsilon(p_0)\big|_{\varepsilon=0}
=\left(\frac15\sin\frac{2\pi}{5},\frac25\sin\frac{2\pi}{5}\right).
\]
Therefore, by the chain rule,
\[
\begin{aligned}
R'(0)
&=
\partial_\varepsilon\tau_\varepsilon(p_0)\big|_{\varepsilon=0}
+\nabla\tau_0(p_0)\cdot p_0'+
\partial_\varepsilon\tau_\varepsilon(q_0)\big|_{\varepsilon=0}
+\nabla\tau_0(q_0)\cdot q_0'\\
&=
-\frac45 \sin\frac{2\pi}{5}
+\frac1\pi\left(\cos\frac{2\pi}{5}-1\right)\neq 0.
\end{aligned}
\]
In particular, we can choose $0<\varepsilon<\delta$ such that $R(\varepsilon)\notin \mathbb Q/\mathbb Z$. This finishes the proof. 
\end{proof}

\bibliographystyle{plain}
\bibliography{ref}
\end{document}